\documentclass[11pt,reqno]{amsart}

\usepackage{amsthm, amsmath, amsfonts, amssymb, graphicx, tikz-cd, hyperref, mathtools}
\usepackage[capitalize,noabbrev]{cleveref} 
\usepackage[shortlabels]{enumitem}

\newtheorem{theorem}{Theorem}[section]
\newtheorem{lemma}[theorem]{Lemma}
\newtheorem{proposition}[theorem]{Proposition}
\newtheorem{corollary}[theorem]{Corollary}

\theoremstyle{definition}
\newtheorem{definition}[theorem]{Definition}
\newtheorem{example}[theorem]{Example}
\newtheorem{notation}[theorem]{Notation}
\newtheorem{remark}[theorem]{Remark}
\newtheorem{claim}[theorem]{Claim}

\crefname{theorem}{Theorem}{Theorems}
\crefname{lemma}{Lemma}{Lemmas}
\crefname{proposition}{Proposition}{Propositions}
\crefname{corollary}{Corollary}{Corollaries}
\crefname{conjecture}{Conjecture}{Conjectures}
\crefname{definition}{Definition}{Definitions}
\crefname{example}{Example}{Examples}
\crefname{question}{Question}{Questions}
\crefname{remark}{Remark}{Remarks}
\crefname{claim}{Claim}{Claims}
\crefname{notation}{Notation}{Notations}

\AddToHook{env/lemma/begin}{\crefalias{theorem}{lemma}}
\AddToHook{env/proposition/begin}{\crefalias{theorem}{proposition}}
\AddToHook{env/corollary/begin}{\crefalias{theorem}{corollary}}
\AddToHook{env/conjecture/begin}{\crefalias{theorem}{conjecture}}
\AddToHook{env/definition/begin}{\crefalias{theorem}{definition}}
\AddToHook{env/example/begin}{\crefalias{theorem}{example}}
\AddToHook{env/question/begin}{\crefalias{theorem}{question}}
\AddToHook{env/remark/begin}{\crefalias{theorem}{remark}}
\AddToHook{env/claim/begin}{\crefalias{theorem}{claim}}
\AddToHook{env/notation/begin}{\crefalias{theorem}{notation}}

\newenvironment{proofclaim}{\paragraph{\emph{Proof of the Claim}.}}{\hfill$\qed$\\}

\newcommand{\Pries}{\mathsf{Pries}}
\newcommand{\ERS}{\mathsf{ERS}}
\newcommand{\Esa}{\mathsf{Esa}}
\newcommand{\CC}{\mathsf{CC}}
\newcommand{\GA}{\mathsf{GA}}
\newcommand{\HA}{\mathsf{HA}}
\newcommand{\DL}{\mathsf{DL}}
\newcommand{\up}{{\uparrow}}
\newcommand{\down}{{\downarrow}}
\newcommand{\V}{\mathsf{V}}
\newcommand{\Vcal}{\mathcal{V}}

\newcommand{\timesg}{\bigotimes_i}
\newcommand{\timesp}{\prod_i}
\newcommand{\upc}{{\Uparrow}}
\newcommand{\downc}{{\Downarrow}}
\newcommand{\lec}{\unlhd}

\newcommand{\nlec}{\ntrianglelefteq}

\newcommand{\upup}{{\twoheaduparrow}}
\newcommand{\bd}{{\mathsf{bd}}}
\newcommand{\two}{{\mathsf{2}}}
\newcommand{\id}{\text{id}}

\DeclareFontFamily{U}{mathx}{\hyphenchar\font45}
\DeclareFontShape{U}{mathx}{m}{n}{
      <5> <6> <7> <8> <9> <10>
      <10.95> <12> <14.4> <17.28> <20.74> <24.88>
      mathx10
      }{}
\DeclareSymbolFont{mathx}{U}{mathx}{m}{n}
\DeclareMathSymbol{\bigtimes}{1}{mathx}{"91}

\DeclareFontFamily{U} {MnSymbolA}{}
\DeclareFontShape{U}{MnSymbolA}{m}{n}{
  <-6> MnSymbolA5
  <6-7> MnSymbolA6
  <7-8> MnSymbolA7
  <8-9> MnSymbolA8
  <9-10> MnSymbolA9
  <10-12> MnSymbolA10
  <12-> MnSymbolA12}{}
\DeclareFontShape{U}{MnSymbolA}{b}{n}{
  <-6> MnSymbolA-Bold5
  <6-7> MnSymbolA-Bold6
  <7-8> MnSymbolA-Bold7
  <8-9> MnSymbolA-Bold8
  <9-10> MnSymbolA-Bold9
  <10-12> MnSymbolA-Bold10
  <12-> MnSymbolA-Bold12}{}

\DeclareSymbolFont{MnSyA} {U} {MnSymbolA}{m}{n}

\DeclareMathSymbol{\uprsquigarrow}{\mathrel}{MnSyA}{161}
\DeclareMathSymbol{\twoheaduparrow}{\mathrel}{MnSyA}{25}

\setlist[enumerate,1]{label={\upshape(\arabic*)},ref=\arabic*}

\makeatletter 
\edef\plabelformat{(\string#2\string#1\string#3)}
\edef\plabelrangeformat{(\string#3\string#1,\string#2\string#6)}
\newcommand{\plabel}[1]{\label{#1}
\immediate\write\@auxout{\noexpand\crefformat{#1}{\noexpand\cref{#1}\plabelformat}
\noexpand\crefmultiformat{#1}{\noexpand\cref{#1}\plabelformat}{,\plabelformat}{,\plabelformat}{,\plabelformat}
\noexpand\crefrangeformat{#1}{\noexpand\cref{#1}\plabelrangeformat}}}
\makeatother

\makeatletter
\@namedef{subjclassname@2020}{\textup{2020} Mathematics Subject Classification}
\makeatother

\pgfdeclarelayer{edgelayer}
\pgfdeclarelayer{nodelayer}
\pgfsetlayers{edgelayer,nodelayer,main}

\tikzstyle{none}=[inner sep=0pt]
\tikzstyle{black dot}=[fill=black, draw=black, shape=circle, minimum size=5pt, inner sep=0]
\tikzstyle{small black dot}=[fill=black, draw=black, shape=circle, minimum size=0.75pt, inner sep=0]
\tikzstyle{white dot}=[fill=white, draw=black, shape=circle, minimum size=5pt, inner sep=0]

\tikzstyle{none dashed}=[dashed, -]
\tikzstyle{none dotted}=[dotted, -, tikzit draw={rgb,255: red,79; green,160; blue,203}]
\tikzstyle{to}=[->]

\begin{document}

\title[Free algebras and coproducts in varieties of G\"odel algebras]{Free algebras and coproducts in varieties\\ of G\"odel algebras}

\author{L.~Carai}
\address{Dipartimento di Matematica ``Federigo Enriques'', Universit\`a degli Studi di Milano, via Cesare Saldini 50, 20133 Milano, Italy}
\email{luca.carai.uni@gmail.com}


\begin{abstract}
G\"odel algebras are the Heyting algebras satisfying the axiom $(x \to y) \vee (y \to x)=1$.
We utilize Priestley and Esakia dualities to dually describe free G\"odel algebras and coproducts of G\"odel algebras.
In particular, we realize the Esakia space dual to a G\"odel algebra free over a distributive lattice as the, suitably topologized and ordered, collection of all nonempty closed chains of the Priestley dual of the lattice.
This provides a tangible dual description of free G\"odel algebras without any restriction on the number of free generators, which generalizes known results for the finitely generated case.
A similar approach allows us to characterize the Esakia spaces dual to coproducts of arbitrary families of G\"odel algebras. 
We also establish analogous dual descriptions of free algebras and coproducts in every variety of G\"odel algebras.
As consequences of these results, we obtain a formula to compute the depth of coproducts of G\"odel algebras and show that all free G\"odel algebras are bi-Heyting algebras. 
\end{abstract}

\maketitle
\tableofcontents

\section{Introduction}

Free Heyting algebras play a fundamental role in the study of intuitionistic propositional logic as they are, up to isomorphism,
Lindenbaum-Tarski algebras, whose elements are equivalence classes of propositional formulas over a fixed set of variables modulo intuitionistic logical equivalence.
The notoriously intricate structure of free Heyting algebras
can be investigated using the powerful tool of Esakia duality, which establishes a dual equivalence between the category of Heyting algebras and a category of ordered topological spaces known as Esakia spaces (see, e.g., \cite{Esa19}).
Different methods to study the Esakia spaces dual to free Heyting algebras have been developed.
Universal models, introduced independently by Shehtman \cite{She78} and Bellissima \cite{Bel86}, constitute the upper part of the Esakia duals of finitely generated free Heyting algebras. The coloring technique due to Esakia and Grigolia \cite{EG77} is one of the main tools to construct universal models (see, e.g., \cite[Sec.~3]{Bez06}).
A different approach, developed by Ghilardi \cite{Ghi92} generalizing some results of Urquhart~\cite{Urq73},
builds the Esakia duals of finitely generated free Heyting algebras as the inverse limits of systems of finite posets (see also \cite{BG11}). This is known as the step-by-step method and it has recently been generalized 
beyond the finitely generated setting by Almeida \cite{Alm24}.

Due to the complexity of free Heyting algebras, it is natural to restrict the attention to free algebras in smaller varieties of Heyting algebras.
A particularly well-behaved variety of Heyting algebras is the variety $\GA$ of G\"odel algebras, which are the Heyting algebras satisfying the prelinearity axiom $(x \to y) \vee (y \to x)=1$.  The variety $\GA$ provides the algebraic semantics for the intermediate propositional logic known as the G\"odel-Dummett logic or linear calculus, introduced by Dummett~\cite{Dum59} and often denoted by $\mathsf{LC}$. The G\"odel-Dummett logic can also be thought of as a propositional fuzzy logic (see, e.g., \cite{BP11} and \cite[Sec.~4.2]{Haj98}). 

Free G\"odel algebras were first studied by Horn \cite{Hor69a}, who proved that $\GA$ is locally finite, meaning that finitely generated free G\"odel algebras are finite. Grigolia \cite{Gri87} described the Esakia duals of finitely generated free G\"odel algebras, while Aguzzoli, Gerla, and Marra \cite{AGM08} described the Esakia duals of G\"odel algebras free over finite distributive lattices. A G\"odel algebra $G$ is said to be free over a distributive lattice $L$ via a lattice homomorphism $e \colon L \to G$ when the following holds: for every G\"odel algebra $H$ and lattice homomorphism $f \colon L \to H$, there is a unique Heyting algebra homomorphism $g \colon G \to H$ such that $g \circ e = f$. 
\[
\begin{tikzcd}[sep = large]
G \arrow[r, dashed, "\exists ! \,g"] & H \\
L \arrow[ur, "f"'] \arrow[u, "e"] & 
\end{tikzcd}
\]  
By Priestley duality, the category of distributive lattices is dually equivalent to the category of the ordered topological spaces known as Priestley spaces (see, e.g., \cite{GvG24}).
The main result of \cref{sec:free Godel} generalizes the results of \cite{AGM08} from the finite to the infinite setting, by providing a dual description of G\"odel algebras free over distributive lattices, without any restriction on the cardinality of the lattice.
We show that the Esakia dual of the G\"odel algebra free over a distributive lattice $L$ is isomorphic to the Esakia space whose points are all the nonempty closed chains (i.e., nonempty totally ordered closed subsets) of the Priestley space dual to $L$.
As a consequence, we obtain that the G\"odel algebra free over a set $S$ is dual to the Esakia space of all nonempty closed chains of $\two^S$, where $\two$ is the $2$-element chain with the discrete topology.
This result is noteworthy because it offers a tangible dual description of free G\"odel algebras, contrasting with the sparsity of concrete descriptions available for the duals of free algebras in varieties of Heyting algebras, especially in the infinitely generated setting.

Coproducts of Heyting algebras, and hence products in the category of Esakia spaces, are notoriously difficult to describe. A generalization of the construction of universal models was employed by Grigolia \cite{Gri06} to study the upper part of the Esakia duals of binary coproducts of finite Heyting algebras. An algorithm to compute binary coproducts of finite G\"odel algebras was presented by D'Antona and Marra in \cite{DM06}. The step-by-step method has been employed in \cite[Sec.~4.2]{Alm24} to obtain a dual description of binary coproducts of Heyting algebras. In \cref{sec:coproducts}, we utilize the machinery developed in \cref{sec:free Godel} to obtain a dual description of coproducts of any family of G\"odel algebras. 
We prove that the Esakia dual of a coproduct is realized as a particular collection of nonempty closed chains of the cartesian product of the Esakia duals of the factors. Notably, this result does not require any restrictions on the cardinalities of the family and of its members.

Dunn and Meyer~\cite{DM71} and Hecht and Katri\v{n}\'{a}k~~\cite{HK72} showed that there are countably many proper subvarieties (i.e., equationally definable subclasses) of $\GA$, and each of them is axiomatized over $\GA$ by the bounded depth axiom $\bd_n = 1$ for some $n \in \mathbb{N}$, where $\bd_n$ is the $n$-ary term defined recursively as follows:
\begin{align*}
\bd_0 & \coloneqq 0,\\
\bd_{n} & \coloneqq x_{n} \vee (x_{n} \to \bd_{n-1}).
\end{align*}
We denote by $\GA_n$ the subvariety of $\GA$ consisting of all the G\"odel algebras validating $\bd_n = 1$, and we refer to its members as $\GA_n$-algebras. 
In particular, $\GA_0$ contains only trivial algebras and $\GA_1$ coincides with the variety of boolean algebras. Since $\GA_n \subseteq \GA_m$ iff $n \le m$, the subvarieties of $\GA$ form a countable chain of order type $\omega+1$. This description of the subvarieties of $\GA$ is the algebraic counterpart of Hosoi's characterization of the extensions of the G\"odel-Dummett logic (see~\cite{Hos67}).
In \cref{sec:free Godel,sec:coproducts} we also adapt the dual descriptions of free G\"odel algebras and coproducts of G\"odel algebras mentioned above to obtain dual descriptions of free $\GA_n$-algebras over distributive lattices and of coproducts in $\GA_n$. 
The concreteness of the dual description of coproducts allows us to obtain in \cref{sec:coproducts} a formula to calculate the depth of a coproduct in $\GA$ from the depths of its factors.

It is shown in \cite{Ghi92} that the step-by-step method allows to conclude that every Heyting algebra free over a finite distributive lattice is a bi-Heyting algebra, where a Heyting algebra is called a bi-Heyting algebra if its order dual is also a Heyting algebra. The G\"odel algebras that are bi-Heyting algebras are also known as bi-G\"odel algebras and have been studied in \cite{BMM22}.
In \cref{sec:bi-Heyting} we show that a G\"odel algebra free over a distributive lattice $L$ is a bi-Heyting algebra iff the order dual of $L$ is a Heyting algebra.  As a consequence, we deduce that free G\"odel algebras are always bi-Heyting algebras. Surprisingly, the situation is very different for free $\GA_n$-algebras. In fact, we prove that free $\GA_n$-algebras are never bi-Heyting algebras, except when they are finitely generated. 

We end the paper with \cref{sec:comparison stepbystep} in which we compare our approach with the description of free G\"odel algebras provided by the step-by-step method.
We also investigate the dual description of some particular sublattices of free G\"odel algebras that play an important role in the step-by-step construction.

\section{Preliminaries on Priestley and Esakia dualities}\label{sec:prelim Priestley Esakia}

In this section we recall the basics of Priestley duality for distributive lattices and of Esakia duality for Heyting algebras. We also describe how Esakia duality restricts to varieties of G\"odel algebras. For more details see, e.g., \cite{GvG24,Esa19}. 

If $X$ is a poset and $A \subseteq X$, we let
\[
\up A \coloneqq \{x \in X \mid y \le x \text{ for some } y \in A \} \qquad \text{and} \qquad \down A \coloneqq \{x \in X \mid x \le y \text{ for some } y \in A \}.
\]
When $A=\{x\}$, we simply write $\up x$ and $\down x$. We call $A$ an \emph{upset} when $A=\up A$ and a \emph{downset} when $A=\down A$.

\begin{definition}
A \emph{Priestley space} is a compact space $X$ equipped with a partial order $\le$ satisfying the \emph{Priestley separation axiom}: if $x,y \in X$ with $x \nleq y$, then there is a clopen upset $U$ such that $x \in U$ and $y \notin U$.
\end{definition}

Each Priestley space is a Stone space (compact, Hausdorff, and zero-dimensional space),
and the topology on finite Priestley spaces is always discrete. So, finite Priestley spaces can be identified with finite posets. We denote by $\Pries$ the category of Priestley spaces and continuous order-preserving maps. Throughout the paper, we will assume that all distributive lattices are bounded and all lattice homomorphisms preserve the bounds. Let $\DL$ be the category of distributive lattices and lattice homomorphisms.
If $X$ is a Priestley space, then the set $X^*$ of clopen upsets of $X$ ordered by inclusion forms a distributive lattice. We then have a contravariant functor $(-)^* \colon \Pries \to \DL$ which sends $X$ to $X^*$ and a $\Pries$-morphism $f \colon X \to Y$ to the $\DL$-morphism $f^{-1} \colon Y^* \to X^*$.
If $L$ is a distributive lattice, we denote by $L_*$ the set of all prime filters of $L$ ordered by inclusion and equipped with the topology generated by the subbasis $\{ \sigma_L(a), \ L_* \setminus \sigma_L(a) \mid a \in L\}$, where $\sigma_L(a)= \{ P \in L_* \mid a \in P\}$. It turns out that $L_*$ is a Priestley space and there is a contravariant functor $(-)_* \colon \DL \to \Pries$ which maps $L$ to $L_*$ and a $\DL$-morphism $\alpha \colon L \to M$ to the $\Pries$-morphism $\alpha^{-1} \colon M_* \to L_*$. These two functors are quasi-inverses of each other and establish Priestley duality. 

\begin{theorem}[Priestley duality]
$\Pries$ is dually equivalent to $\DL$.
\end{theorem}

On the one hand, the map $\sigma_L \colon L \to (L_*)^*$ is an isomorphism for every $L \in \DL$, and yields a natural isomorphism $\sigma \colon \id_{\DL} \to (-)^* \circ (-)_*$. On the other hand, for every $X \in \Pries$ the map $\varepsilon_X \colon X \to (X^*)_*$ sending $x$ to $\{U \in X^* \mid x \in U\}$ is an isomorphism of Priestley spaces and yields a natural isomorphism $\varepsilon \colon \id_{\Pries} \to (-)_* \circ (-)^*$.

Recall that a distributive lattice $H$ is called a \emph{Heyting algebra} when it is equipped with a binary operation $\to$ called \emph{implication} such that $a \wedge b \le c$ iff $a \le b \to c$ for every $a,b,c \in H$. 
Let $\HA$ be the category of Heyting algebras and Heyting homomorphisms; that is, lattice homomorphisms preserving implications.
We now turn our attention to Esakia duality for Heyting algebras.

\begin{definition}
A Priestley space $X$ is called an \emph{Esakia space} when $\down V$ is clopen for every clopen subset $V$ of $X$.
\end{definition}

A map $f \colon X \to Y$ between posets is called a \emph{p-morphism} if $f[\up x]=\up f(x)$ for every $x \in X$. Equivalently, a p-morphism is an order-preserving map such that $f(x) \le y$ implies the existence of $z \ge x$ such that $f(z)=y$. Esakia spaces and continuous p-morphisms form a category that we denote by $\Esa$. To help the reader, we will usually denote Priestley spaces with the letter $X$ and Esakia spaces with $Y$. 
If $Y$ is an Esakia space, then $Y^*$ forms a Heyting algebra with implication given by $U \to V = Y \setminus \down (U \setminus V)$ for every $U,V \in Y^*$.
In fact, the functors $(-)^*$ and $(-)_*$ of Priestley duality restrict to $\HA$ and $\Esa$ and yield Esakia duality.

\begin{theorem}[Esakia duality]
$\Esa$ is dually equivalent to $\HA$.
\end{theorem}

As mentioned in the introduction, a Heyting algebra $G$ is called a \emph{G\"odel algebra} if the identity $(a \to b) \vee (b \to a)=1$ holds for every $a,b \in G$, where $1$ denotes the greatest element of $G$. We think of the variety $\GA$ of G\"odel algebras as a full subcategory of $\HA$ and consider the restriction of Esakia duality to $\GA$. Recall that a totally ordered subset of a poset $X$ is called a \emph{chain} of $X$, and $X$ is a \emph{root system} if $\up x$ is a chain for every $x \in X$. The following is a consequence of \cite[Thm.~2.4]{Hor69a}.

\begin{theorem}
Let $Y$ be an Esakia space. Then $Y^*$ is a G\"odel algebra iff $Y$ is a root system.
\end{theorem}

We call an Esakia space that is a root system an \emph{Esakia root system}. Let $\ERS$ be the full subcategory of $\Esa$ consisting of Esakia root systems and continuous p-morphisms. It is a straightforward consequence of the previous theorem that Esakia duality restricts to a duality for G\"odel algebras.

\begin{theorem}[Esakia duality for G\"odel algebras]
$\ERS$ is dually equivalent to $\GA$.
\end{theorem}

We end the section with some considerations on the restrictions of Esakia duality to subvarieties of $\GA$.
We first recall the notion of \emph{depth} of an Esakia space, specialized to the setting of Esakia root systems.

\begin{definition}
Let $Y$ be an Esakia root system and $y \in Y$. If $\up y$ is finite, then we denote by $d(y)$ the cardinality of $\up y$. Otherwise, we write $d(y)=\infty$. We call $d(y)$ the \emph{depth} of $y$. We write $d(Y)$ for the supremum of the depths of elements of $Y$, and call $d(Y)$ the depth of $Y$.
\end{definition}

For every $n \in \mathbb{N}$, let $\ERS_n$ be the full subcategory of $\ERS$ consisting of the Esakia root systems of depth smaller or equal to $n$. It follows from well-known facts (see, e.g., \cite[Sec.~2.2]{BBMS21} and the references therein) that a G\"odel algebra $G$ is in $\GA_n$ iff $G_* \in \ERS_n$. We then have the following restriction of Esakia duality.

\begin{theorem}[Esakia duality for $\GA_n$]\label{thm:Esakia duality GAn}
$\ERS_n$ is dually equivalent to $\GA_n$.
\end{theorem}

Since $\GA_n$ is a subvariety of $\GA$, the inclusion functor $\GA_n \hookrightarrow \GA$ has a left adjoint (see, e.g., \cite[Sec.~V.6]{Mac71}).
It then follows from Esakia duality that the inclusion $\ERS_n \hookrightarrow \ERS$ has a right adjoint. We provide a description of such right adjoint. If $Y \in \ERS$, consider the set $Y_n=\{ y \in Y \mid d(y) \le n\}$ with the subspace topology and order induced by $Y$. If $f \colon Y \to Z$ is an $\ERS$-morphism, let $f_n \colon Y_n \to Z_n$ be the restriction of $f$.

\begin{theorem}\label{prop:ERSn to ER right adjoint}
$(-)_n \colon \ERS \to \ERS_n$ is a well-defined functor that is right adjoint to the inclusion $\ERS_n \hookrightarrow \ERS$.
\end{theorem}

\begin{proof}
It follows from \cite[Lem.~7]{Bez00}\footnote{The lemma in the reference is stated for an Esakia space of finite depth, but that assumption is never used in the proof.} that $Y_n$ is a closed upset in $Y$. So, \cite[Lem.~3.4.11]{Esa19} yields that $Y_n$ is an Esakia space.
It is then clear that $Y_n \in \ERS_n$. 
Since $Y_n$ is a closed upset of $Y$, the inclusion $e \colon Y_n \hookrightarrow Y$ is an $\ERS$-morphism. Let $f \colon Z \to Y$ be an $\ERS$-morphism with $Z \in \ERS_n$. We show that there is a unique continuous p-morphism $g \colon Z \to Y_n$ such that $e \circ g = f$. 
If $z \in Z$, then $d(z) \le n$, and so $\up f(z) = f[\up z]$ has cardinality smaller or equal to $n$. Thus, $f[Z] \subseteq Y_n$. Take $g \colon Z \to Y_n$ to be the restriction of $f$. Then $g$ is the unique continuous p-morphism such that $e \circ g = f$. 
It follows from \cite[Thm.~IV.1.2]{Mac71} that $(-)_n$ is a well-defined functor that is right adjoint to the inclusion $\ERS_n \hookrightarrow \ERS$.
\end{proof}

Since $(-)_n$ is a right adjoint, the following is immediate.

\begin{corollary}\label{cor:subn preserving limits}
$(-)_n \colon \ERS \to \ERS_n$ preserves all products.
\end{corollary}

\section{Free G\"odel algebras}\label{sec:free Godel}

In this section we employ Priestley and Esakia dualities to establish a dual description of the G\"odel algebra free over a given distributive lattice. We begin by introducing the notion of closed chain, which will play a fundamental role in this investigation.

\begin{definition}
Let $X$ be a Priestley space. A subset of $X$ is called a \emph{chain} if it is totally ordered with respect to the order on $X$. A chain is said to be \emph{closed} when it is closed in the topology on $X$. We denote by $\CC(X)$ the set of all nonempty closed chains of $X$. 
\end{definition}

Our first goal is to equip $\CC(X)$ with the structure of an Esakia root system and show that $\CC(X)$ is dual to the G\"odel algebra free over the distributive lattice $X^*$. We start by putting a topology on $\CC(X)$. 

If $X$ is a Stone space, let $\V(X)$ be the set of all nonempty closed subsets of $X$. It is well known (see \cite[Thm.~4.9]{Mic51}) that $\V(X)$ becomes a Stone space once equipped with the topology generated by the subbasis $\{ \Box V, \Diamond V \mid V \text{ clopen of } X\}$, where
\begin{align*}
\Box V \coloneqq \{ F \in \V(X) \mid F \subseteq V \} \qquad \text{and} \qquad \Diamond V \coloneqq \{ F \in \V(X) \mid F \cap V \neq \varnothing \}.
\end{align*}
Moreover, $\Box V$ and $\Diamond V$ are clopen subsets of $\V(X)$ for any $V$ clopen of $X$.
The space $\V(X)$ is known as the \emph{Vietoris space} of $X$.

\begin{theorem}\label{thm:CC(X) closed in V(X)}
Let $X$ be a Priestley space. Then $\CC(X)$ is a closed subset of $\V(X)$.
\end{theorem}

\begin{proof}
As the elements of $\CC(X)$ are closed chains, it is clear that $\CC(X) \subseteq \V(X)$. To show that $\CC(X)$ is closed, let $F \in \V(X) \setminus \CC(X)$. Our goal is to exhibit an open neighborhood of $F$ in $\V(X)$ that is disjoint from $\CC(X)$. Since $F \in \V(X) \setminus \CC(X)$, it is a nonempty closed subset of $X$ that is not totally ordered with respect to the order on $X$. Thus, there are $x_1,x_2 \in F$ such that $x_1 \nleq x_2$ and $x_2 \nleq x_1$. Since $X$ is a Priestley space, there are clopen upsets $U_1,U_2$ of $X$ such that $x_1 \in U_1 \setminus U_2$ and $x_2 \in U_2 \setminus  U_1$. Consider $\Vcal \coloneqq \Diamond (U_1 \setminus U_2) \cap \Diamond (U_2 \setminus U_1)$, which is a clopen subset of $\V(X)$ because $U_1 \setminus U_2$ and $U_2 \setminus U_1$ are clopen in $X$. Moreover, $F \in \Vcal$ because $x_1 \in F \cap (U_1 \setminus U_2)$ and $x_2 \in F \cap (U_2 \setminus U_1)$. It remains to show that $\Vcal$ is disjoint from $\CC(X)$. Assume that there is $C \in \CC(X) \cap \Vcal$. Then $C \in \Diamond(U_1 \setminus U_2) \cap \Diamond(U_2\setminus U_1)$, and so there are $y_1,y_2 \in X$ such that $y_1 \in C \cap (U_1  \setminus  U_2)$ and $y_2 \in C \cap (U_2  \setminus  U_1)$. Since $C$ is a chain, $y_1 \le y_2$ or $y_2 \le y_1$. If $y_1 \le y_2$, it follows that $y_2 \in U_1$ because $y_1 \in U_1$ and $U_1$ is an upset. This contradicts that $y_2 \in U_2  \setminus U_1$. If $y_2 \le y_1$, we also obtain a contradiction with a similar argument. Therefore, $\Vcal$ is an open neighborhood of $F$ in $\V(X)$ that is disjoint from $\CC(X)$. As $F$ was an arbitrary element of $\V(X) \setminus \CC(X)$, we have shown that $\CC(X)$ is a closed subset of $\V(X)$. 
\end{proof}

From now on, we fix a Priestley space $X$ and will always assume that $\CC(X)$ is equipped with the subspace topology induced by the Vietoris topology on $\V(X)$.  
The following corollary is an immediate consequence of the fact that a closed subspace of a Stone space is a Stone space (see, e.g., \cite[Lem.~32.2]{GH09}).

\begin{corollary}\label{thm:CC(X) Stone}
$\CC(X)$ is a Stone space.
\end{corollary}

Since $\V(X)$ is leaving the scene and the spotlight will be on $\CC(X)$, with a slight abuse of notation we set
\begin{align*}
\Box A \coloneqq \{ C \in \CC(X) \mid C \subseteq A \} \qquad \text{and} \qquad \Diamond A \coloneqq \{ C \in \CC(X) \mid C \cap A \neq \varnothing \}
\end{align*}
for any subset $A$ of $X$.
In the following lemma we gather some useful facts about the topology on $\CC(X)$ that will be used throughout the paper.

\begin{lemma}\plabel{lem:facts CC(X)}
\hfill\begin{enumerate}
\item\label[lem:facts CC(X)]{lem:facts CC(X):item1} If $A,B \subseteq X$, then $\Box (A \cap B)=\Box A \cap \Box B$ and $\Diamond (A \cup B)=\Diamond A \cup \Diamond B$.
\item\label[lem:facts CC(X)]{lem:facts CC(X):item2} If $A \subseteq X$, then $\CC(X)  \setminus  \Box A = \Diamond (X \setminus A)$ and $\CC(X)  \setminus  \Diamond A = \Box (X \setminus A)$.
\item\label[lem:facts CC(X)]{lem:facts CC(X):item3} If $V$ is clopen in $X$, then $\Box V, \Diamond V$ are clopen in $\CC(X)$.
\item\label[lem:facts CC(X)]{lem:facts CC(X):item4} $\{\Box V,\, \Diamond V \mid V \text{ is clopen in } X\}$ is a subbasis for the topology on $\CC(X)$.
\item\label[lem:facts CC(X)]{lem:facts CC(X):item5} A basis for the topology on $\CC(X)$ is given by the clopen subsets of $\CC(X)$ of the form 
$\Box V \cap \Diamond W_1 \cap \dots \cap \Diamond W_n$ with $V,W_1,\dots, W_n$ clopen in $X$ and $W_1, \dots, W_n \subseteq V$.
\end{enumerate}
\end{lemma}

\begin{proof}
\eqref{lem:facts CC(X):item1} and \eqref{lem:facts CC(X):item2} are straightforward consequences of the definitions of $\Box A$ and $\Diamond A$.

\eqref{lem:facts CC(X):item3} and \eqref{lem:facts CC(X):item4} follow immediately from the definition of the topology on $\CC(X)$.

To verify \eqref{lem:facts CC(X):item5}, observe that it follows from \eqref{lem:facts CC(X):item4} that every open subset of $\CC(X)$ is a union of subsets of the form $\Box V_1 \cap \dots \cap \Box V_m \cap \Diamond V_1' \cap \dots \cap \Diamond V_n'$, where each $V_i$ and $V_j'$ is a clopen subset of $X$. By~\eqref{lem:facts CC(X):item1}, we obtain that $\Box V_1 \cap \dots \cap \Box V_m = \Box V$, where $V \coloneqq V_1 \cap \dots \cap V_m$ is clopen in $X$. Finally, it follows from the definitions of $\Box A$ and $\Diamond B$ that $\Box A \cap \Diamond B = \Box A \cap \Diamond (A \cap B)$ for every $A,B \subseteq X$, and hence
\[
\Box V \cap \Diamond V_1' \cap \dots \cap \Diamond V_n' = \Box V \cap \Diamond (V \cap V_1') \cap \dots \cap \Diamond (V \cap V_n').
\]
This yields the claim because each $V \cap V_j'$ is a clopen subset of $X$ contained in $V$.
\end{proof}

We now define a partial order on $\CC(X)$ that will make it into an Esakia root system.

\begin{definition}
Let $C_1,C_2 \in \CC(X)$. We write $C_1 \lec C_2$ iff $C_2 \subseteq C_1$ and $C_2$ is an upset in $C_1$.
\end{definition}

If $X \in \Pries$, then $\up x$ and $\down x$ are closed for every $x \in X$ (see, e.g., \cite[Prop.~2.6(ii)]{Pri84}). In particular, if $Y \in \ERS$, then $\up y$ is a closed chain for every $y \in Y$. The definition of $\lec$ is inspired by the fact that when $Y$ is an Esakia root system and $y_1,y_2 \in Y$, we have $\up y_1 \lec \up y_2$ iff $y_1 \le y_2$.

It is well known that any nonempty closed subset $F$ of a Priestley space contains elements that are minimal and elements that are maximal in $F$ with respect to $\le$ (see, e.g.,~\cite[Prop.~2.6]{Pri84}\footnote{A proof can be found in \cite[Cor.~3.2.2]{Esa19}, where this fact is stated for Esakia spaces, but the proof works verbatim for Priestley spaces.}).
As an immediate consequence of this fact, we obtain:

\begin{proposition}\label{lem:CC in Pries min and max}
Every $C \in \CC(X)$ has a least and a greatest element.
\end{proposition}

\cref{lem:CC in Pries min and max} yields the following alternative description of $\lec$.

\begin{lemma}\label{lem:equivalent def lec}
$C_1 \lec C_2$ iff there is $x \in C_1$ such that $C_2 = \up x \cap C_1$.
\end{lemma}

\begin{proof}
If $C_2 = \up x \cap C_1$, then $C_2$ is an upset of $C_1$, and hence $C_1 \lec C_2$. To show the other implication, assume $C_1 \lec C_2$ and let $x$ be the least element of $C_2$, which exists by \cref{lem:CC in Pries min and max}. Since $C_1 \lec C_2$, it follows that $C_2 \subseteq C_1$, and hence $C_2 \subseteq \up x \cap C_1$. Conversely, if $y \in \up x \cap C_1$, then $y \in C_2$ because $C_2$ is an upset in $C_1$ and $x \in C_2$. Thus, $\up x \cap C_1 \subseteq C_2$. 
\end{proof}

\begin{theorem}\label{thm:CC(X) root system}
$(\CC(X), \lec)$ is a root system.
\end{theorem}

\begin{proof}
It is an immediate consequence of its definition that $\lec$ is reflexive and antisymmetric. To show that $\lec$ is transitive, let $C_1,C_2,C_3 \in \CC(X)$ such that $C_1 \lec C_2$ and $C_2 \lec C_3$. Then $C_3 \subseteq C_2 \subseteq C_1$ and $C_3$ is an upset in $C_2$, which is an upset in $C_1$. Thus, $C_3$ is an upset in $C_1$, and hence $C_1 \lec C_3$. Thus, $\lec$ is a partial order. Consider $C,C_1,C_2 \in \CC(X)$ with $C \lec C_1,C_2$. We need to show that $C_1 \lec C_2$ or $C_2 \lec C_1$. By \cref{lem:equivalent def lec}, there are $x_1,x_2 \in C$ such that $C_1 = \up x_1 \cap C$ and $C_2 = \up x_2 \cap C$. Since $C$ is a chain, $x_1 \le x_2$ or $x_2 \le x_1$. Therefore, $C_1 \lec C_2$ or $C_2 \lec C_1$. This proves that $(\CC(X), \lec)$ is a root system.
\end{proof}

\begin{notation}
To avoid confusion, subsets of $\CC(X)$ will be denoted with calligraphic capital letters and we will write $\upc \mathcal{A}$ and $\downc \mathcal{A}$ to denote the downset and upset in $(\CC(X), \lec)$ generated by a subset $\mathcal{A}$ of $\CC(X)$. As usual, if $C \in \CC(X)$, we will simply write $\upc C$ and $\downc C$ instead of $\upc \{C\}$ and $\downc \{C\}$.
\end{notation}

From now on, we will always assume that $\CC(X)$ is equipped with $\lec$. To show that $\CC(X)$ is an Esakia root system, we need the following technical lemma.

\begin{lemma}\plabel{lem:Box and Diamond properties}
Let $A, B_1, \dots, B_n$ be subsets of $X$.
\begin{enumerate}
\item\label[lem:Box and Diamond properties]{lem:Box and Diamond properties:item1} $\Box A$ is an upset of $\CC(X)$.
\item\label[lem:Box and Diamond properties]{lem:Box and Diamond properties:item2} $\Diamond A$ is a downset of $\CC(X)$.
\item\label[lem:Box and Diamond properties]{lem:Box and Diamond properties:item3} $\downc (\Box A \cap \Diamond B_1 \cap \dots \cap \Diamond B_n) = \downc (\Box A \cap \Diamond B_1) \cap \dots \cap \downc (\Box A \cap \Diamond B_n)$.
\end{enumerate}
Let also $D$ be a downset of $X$ and $U$ an upset of $X$.
\begin{enumerate}[resume]
\item\label[lem:Box and Diamond properties]{lem:Box and Diamond properties:item4} $\Box D$ is a downset of $\CC(X)$.
\item\label[lem:Box and Diamond properties]{lem:Box and Diamond properties:item5} $\Diamond U$ is an upset of $\CC(X)$.
\item\label[lem:Box and Diamond properties]{lem:Box and Diamond properties:item6} If  $U \cap D \subseteq A$, then 
\[
\downc (\Box A \cap \Diamond (U \cap D)) = \Box(A \cup D) \cap \Diamond (U \cap D).
\]
\item\label[lem:Box and Diamond properties]{lem:Box and Diamond properties:item7} If $D \subseteq A$, then $\Box A \cap \Diamond D$ is a downset of $\CC(X)$.
\end{enumerate}
\end{lemma}

\begin{proof}
\eqref{lem:Box and Diamond properties:item1}. Let $C_1,C_2 \in \CC(X)$ such that $C_1 \lec C_2$ and $C_1 \in \Box A$. Then $C_2 \subseteq C_1$ and $C_1 \subseteq A$. Thus, $C_2 \subseteq A$, and hence $C_2 \in \Box A$. This shows that $\Box A$ is an upset.

\eqref{lem:Box and Diamond properties:item2}. It follows from \eqref{lem:Box and Diamond properties:item1} that $\Box (X \setminus A)$ is an upset. By \cref{lem:facts CC(X):item2}, we have that $\Diamond A = \CC(X)  \setminus \Box (X \setminus A)$, and so $\Diamond A$ is a downset because it is the complement of an upset.

\eqref{lem:Box and Diamond properties:item3}. 
The left-to-right inclusion is an immediate consequence of the fact that taking the downset of a subset preserves inclusions. To prove the other inclusion, let 
\[
C \in \downc (\Box A \cap \Diamond B_1) \cap \dots \cap \downc (\Box A \cap \Diamond B_1).
\]	
Then there are $C_1, \dots, C_n \in \CC(X)$ such that $C \lec C_i$ and $C_i \in \Box A \cap \Diamond B_i$ for each $i$.
By \cref{thm:CC(X) root system}, $\CC(X)$ is a root system, and hence $\upc C$ is totally ordered with respect to $\lec$. Since $C_1,\dots,C_n \in \upc C$, there is $j$ such that $C_j \lec C_i$ for every $i$.
By \eqref{lem:Box and Diamond properties:item2}, $\Diamond B_i$ is a downset, and so $C_j \in \Diamond B_i$ for every $i$ because $C_i \in \Diamond B_i$ and $C_j \lec C_i$. Then $C_j \in \Box A \cap \Diamond B_1 \cap \dots \cap \Diamond B_n$.
Therefore, $C \in \downc (\Box A \cap \Diamond B_1 \cap \dots \cap \Diamond B_n)$ since $C \lec C_j$.

\eqref{lem:Box and Diamond properties:item4}. Let $C_1,C_2 \in \CC(X)$ with $C_1 \lec C_2$ and $C_2 \in \Box D$. We want to show that $C_1 \in \Box D$. Since $C_1 \lec C_2$, there is $x \in C_1$ such that $C_2 = \up x \cap C_1$ by \cref{lem:equivalent def lec}.  Let $y \in C_1$. 
Then $x \le y$ or $y \le x$ because $x,y \in C_1$ and $C_1$ is a chain.
If $x \le y$, we have $y \in \up x \cap C_1 = C_2 \subseteq D$. Otherwise, $y \le x \in C_2 \subseteq D$, and hence $y \in D$ because $D$ is a downset. In either case, $y \in D$. This shows that $C_1 \subseteq D$, and so we have proved that $C_1 \in \Box D$.

\eqref{lem:Box and Diamond properties:item5}. Since $U$ is an upset of $X$, we have that $X \setminus U$ is a downset, and so $\Box (X \setminus U)$ is a downset of $\CC(X)$ by \eqref{lem:Box and Diamond properties:item4}. \cref{lem:facts CC(X):item2} implies that $\Diamond U = \CC(X)  \setminus \Box (X \setminus U)$. Thus, $\Diamond U$ is an upset because it is the complement of a downset.

\eqref{lem:Box and Diamond properties:item6}. To show the left-to-right inclusion, assume that $C \in \downc (\Box A \cap \Diamond (U \cap D))$. Then there is $K \in \CC(X)$ such that $C \lec K$ and $K \in \Box A \cap \Diamond (U \cap D)$. Since $\Diamond (U \cap D)$ is a downset by \eqref{lem:Box and Diamond properties:item2}, we have that $C \in \Diamond (U \cap D)$. It then remains to show that $C \in \Box(A \cup D)$. From $K \in \Box A \cap \Diamond (U \cap D)$ it follows that $K \subseteq A$ and $K \cap U \cap D \neq \varnothing$.  Let $x \in K \cap U \cap D$. 
Since $C \lec K$ and $x \in K$, we obtain $\up x \cap C \subseteq K$. So, $\up x \cap C \subseteq A$ because $K \subseteq A$. We also have that $\down x \cap C \subseteq D$ because $D$ is a downset and $x \in D$. That $C$ is a chain implies $C = (\up x \cap C) \cup (\down x \cap C)$. Therefore, $C = (\up x \cap C) \cup (\down x \cap C) \subseteq A \cup D$, and hence $C \in \Box (A \cup D)$. This shows that 
$C \in \Box(A \cup D) \cap \Diamond (U \cap D)$. 

We now prove the other inclusion. Suppose that $C \in \Box(A \cup D) \cap \Diamond (U \cap D)$. Then $C \subseteq A \cup D$ and $C \cap U \cap D \neq \varnothing$. Take $x \in C \cap U \cap D$ and let $K = \up x \cap C$. Thus, $K \in \CC(X)$ and $C \lec K$. We show that $K \in \Box A \cap \Diamond (U \cap D)$. Since $x \in K \cap U \cap D$, we have that $K \in \Diamond (U \cap D)$. By \eqref{lem:Box and Diamond properties:item1},
$\Box (A\cup D)$ is an upset. Then $C \lec K$ implies
$K \in \Box(A \cup D)$, and so $K \subseteq A \cup D$. Because $x \in U$ and $U$ is an upset, we have $K \subseteq \up x \subseteq U$. Then the hypothesis that $U \cap D \subseteq A$ yields
\[
K \subseteq (A \cup D) \cap U \subseteq A \cup (U \cap D) = A,
\]
which implies that $K \in \Box A$. Therefore, $K \in \Box A \cap \Diamond (U \cap D)$, and so $C \in \downc (\Box A \cap \Diamond (U \cap D))$ because $C \lec K$.

\eqref{lem:Box and Diamond properties:item7}. It immediately follows from \eqref{lem:Box and Diamond properties:item6} by taking $U=X$.
\end{proof}

We will also need the following well-known property of Priestley spaces.

\begin{lemma}{\cite[Lem.~11.22]{DP02}}\label{prop:clopen union of convex}
Every clopen subset of a Priestley space $X$ is a finite union of subsets of the form $U \cap D$, where $U$ is a clopen upset and $D$ a clopen downset.
\end{lemma}

\begin{theorem}\label{thm:CC in ERS}
$\CC(X)$ is an Esakia root system.
\end{theorem}

\begin{proof}
\cref{thm:CC(X) Stone,thm:CC(X) root system} yield that $\CC(X)$ is a Stone space and a root system. It remains to show that it is an Esakia space.
We first prove that $\CC(X)$ is a Priestley space. Let $C_1,C_2 \in \CC(X)$ with $C_1 \nlec C_2$. 
First, consider the case in which $C_2 \nsubseteq C_1$. Then there is $x \in C_2 \setminus C_1$. Since $X$ is a Stone space and $C_1$ is closed, we can find a clopen subset $V$ of $X$ such that $C_1 \subseteq V$ and $x \notin V$. Thus, $C_1 \in \Box V$ and $C_2 \notin \Box V$. By \cref{lem:Box and Diamond properties:item1}, $\Box V$ is a clopen upset of $\CC(X)$ containing $C_1$ and not $C_2$. Let us now assume that $C_2 \subseteq C_1$. By \cref{lem:CC in Pries min and max}, there exist the greatest elements $x_1$ and $x_2$ of $C_1$ and $C_2$, respectively. Since $C_2 \subseteq C_1$, it must be that $x_2 < x_1$ or $x_1 = x_2$. If $x_2 < x_1$, then there is a clopen upset $U$ of $X$ such that $x_1 \in U$ and $x_2 \notin U$ because $X$ is a Priestley space. Then $x_1 \in C_1 \cap U$ and $C_2 \cap U = \varnothing$ because $x_2$ is the greatest element of $C_2$ and $U$ is an upset. By \cref{lem:Box and Diamond properties:item5}, $\Diamond U$ is a clopen upset of $\CC(X)$ such that $C_1 \in \Diamond U$ and $C_2 \notin \Diamond U$. It then remains to consider the case in which $x_1=x_2$. Since $C_2 \subseteq C_1$ and $C_1 \nlec C_2$, there is $y \in C_1  \setminus  C_2$ such that $\down y \cap C_2 \neq \varnothing$. Then $x_2 \in \up y \cap C_2$ because $y \in C_1$ and $x_2$ coincides with the greatest element of $C_1$. Thus, $\down y \cap C_2$ and $\up y \cap C_2$ are nonempty closed chains in $X$. Let $z_1$ and $z_2$ be the greatest and least elements of $\down y \cap C_2$ and $\up y \cap C_2$, respectively. Since $y \notin C_2$, we have $y \nleq z_1$ and $z_2 \nleq y$. Because $X$ is a Priestley space, there is a clopen downset $D$ such that $z_1 \in D$ and $y \notin D$ and there is a clopen upset $U$ such that $z_2 \in U$ and $y \notin U$. Then $C_1 \notin \Box (U \cup D) \cap \Diamond D$ because $y \in C_1$ and $y \notin U \cup D$. 
Since $C_2$ is a chain and $z_1,z_2$ are the greatest and least elements of $\down y \cap C_2$ and $\up y \cap C_2$, we have that $C_2 \subseteq (\up y \cap C_2) \cup (\down y \cap C_2) \subseteq \up z_2 \cup \down z_1$. So,
$C_2 \in \Box (U \cup D) \cap \Diamond D$ because $C_2 \subseteq \up z_2 \cup \down z_1 \subseteq U \cup D$ and $z_1 \in C_2 \cap D$.
By \cref{lem:Box and Diamond properties:item7}, $\mathcal{U} \coloneqq \CC(X) \setminus (\Box (U \cup D) \cap \Diamond D)$ is a clopen upset of $\CC(X)$ such that $C_1 \in \mathcal{U}$ and $C_2 \notin \mathcal{U}$.
We have proved that  $\CC(X)$ is a Priestley space

To prove that $\CC(X)$ is an Esakia space, we need to show that the downset of any clopen subset of $\CC(X)$ is clopen. By \cref{lem:facts CC(X):item5}, any clopen of $\CC(X)$ can be written as a finite union of clopens of the form $\Box V \cap \Diamond W_1 \cap \cdots \cap \Diamond W_n$ with $V,W_1, \dots, W_n$ clopen subsets of $X$. Since taking the downset commutes with unions, it is enough to show that $\downc (\Box V \cap \Diamond W_1 \cap \dots \cap \Diamond W_n)$ is clopen for every $V,W_1, \dots, W_n$ clopens of $X$.
\cref{lem:Box and Diamond properties:item3} implies that it is sufficient to show that $\downc (\Box V \cap \Diamond W)$ is clopen for any $V,W$ clopens of $X$. By \cref{prop:clopen union of convex}, if $V$ and $W$ are clopens, then the clopen $V \cap W$ can be written as 
\begin{equation}\label{eq:V cap W = cup Ui cap Di}
V \cap W=(U_1 \cap D_1) \cup \dots \cup (U_m \cap D_m),
\end{equation}
where $U_i$ and $D_i$ are, respectively, a clopen upset and a clopen downset of $X$ for each $i$. 
We obtain
\begin{align*}
\Box V \cap \Diamond W &= \Box V \cap \Diamond (V \cap W) = \Box V \cap \Diamond ((U_1 \cap D_1) \cup \dots \cup (U_1 \cap D_m))\\
&= \Box V \cap (\Diamond (U_1 \cap D_1) \cup \dots \cup \Diamond (U_m \cap D_m))\\
&= (\Box V \cap \Diamond (U_1 \cap D_1)) \cup \dots \cup (\Box V \cap \Diamond (U_m \cap D_m)),
\end{align*}
where the first equality follows from the fact that $\Box A \cap \Diamond B = \Box A \cap \Diamond (A \cap B)$ for every $A,B \subseteq X$, the second from \cref{eq:V cap W = cup Ui cap Di}, the third from \cref{lem:facts CC(X):item1}, and the fourth is straightforward. Therefore,
\begin{align*}
\downc (\Box V \cap \Diamond W) = \downc (\Box V \cap \Diamond (U_1 \cap D_1)) \cup \dots \cup \downc (\Box V \cap \Diamond (U_m \cap D_m)).
\end{align*}
Since $U_i \cap D_i \subseteq V$ for every $i$, \cref{lem:Box and Diamond properties:item6} yields that 
\[
\downc (\Box V \cap \Diamond (U_i \cap D_i))= \Box (V \cup D_i) \cap \Diamond (U_i \cap D_i), 
\]
which is a clopen subset of $\CC(X)$. Thus, $\downc (\Box V \cap \Diamond W)$ is clopen. We have shown that the downset of any clopen of $\CC(X)$ is clopen, and hence that $\CC(X)$ is an Esakia space. 
\end{proof}

We have seen that $\CC(X)$ is an Esakia root system. Our next goal is to prove that its dual G\"odel algebra $\CC(X)^*$ is free over the distributive lattice $X^*$. This we achieve by showing that $\CC(X)$ satisfies the universal property dual to the universal property of free G\"odel algebras.
Let $m \colon \CC(X) \to X$ be the map that sends each $C \in \CC(X)$ to its least element, which exists by \cref{lem:CC in Pries min and max}.

\begin{lemma}\plabel{lem:minv of up and down}
\hfill\begin{enumerate}
\item\label[lem:minv of up and down]{lem:minv of up and down:item1} If $U \subseteq X$ is an upset, then $m^{-1}[U]=\Box U$.
\item\label[lem:minv of up and down]{lem:minv of up and down:item2} If $D \subseteq X$ is a downset, then $m^{-1}[D]=\Diamond D$.
\end{enumerate}
\end{lemma}

\begin{proof}
\eqref{lem:minv of up and down:item1}. Let $U$ be an upset of $X$ and $C \in \CC(X)$. Since $C \subseteq \up m(C)$, we have that $m(C) \in U$ iff $C \subseteq U$. Thus, $m^{-1}[U]=\Box U$.

\eqref{lem:minv of up and down:item2}. Let $D$ be a downset of $X$ and $C \in \CC(X)$. Since $m(C) \in \down x$ for every $x \in C$, we have that $m(C) \in D$ iff $C \cap D \neq \varnothing$. Thus, $m^{-1}[D]=\Diamond D$.
\end{proof}

\begin{lemma}\label{lem:m Priestley map}
The map $m \colon \CC(X) \to X$ is continuous and order preserving.
\end{lemma}

\begin{proof}
We first show that $m$ is continuous. Since $X$ is a Stone space, it is enough to prove that $m^{-1}[V]$ is clopen for each clopen subset $V$ of $X$. Since $m^{-1}$ commutes with unions and intersections, \cref{prop:clopen union of convex} implies that it is sufficient to show that $m^{-1}[U]$ and $m^{-1}[D]$ are clopen in $\CC(X)$ for each $U$ clopen upset and $D$ clopen downset of $X$. By \cref{lem:minv of up and down}, if $U$ is a clopen upset and $D$ a clopen downset, then $m^{-1}[U] = \Box U$ and $m^{-1}[D] = \Diamond D$, which are clopen in $\CC(X)$. Thus, $m$ is continuous.

It remains to show that $m$ is order preserving. Let $C_1,C_2 \in \CC(X)$ with $C_1 \lec C_2$. Then there is $x \in C_1$ such that $C_2 = \up x \cap C_1$. Thus, $m(C_2) = x \in C_1$, and hence $m(C_1) \le m(C_2)$. Therefore, $m$ is order preserving.
\end{proof}

Since every continuous map between Stone spaces is a closed map and the image of a chain under an order-preserving map is a chain, the following lemma is immediate.

\begin{lemma}\label{lem:image of chain is chain}
If $f \colon X_1 \to X_2$ is a continuous order-preserving map between Priestley spaces and $C \in \CC(X_1)$, then $f[C] \in \CC(X_2)$. 
\end{lemma}

We are finally ready to prove the universal property of $\CC(X)$.

\begin{theorem}\label{thm:univ prop CC(X)}
Let $Y$ be an Esakia root system and $f \colon Y \to X$ an order-preserving continuous map. Then there is a unique continuous p-morphism $g \colon Y \to \CC(X)$ such that $m \circ g = f$.
\[
\begin{tikzcd}[sep = large]
\CC(X) \arrow[d, "m"'] & Y \arrow[ld, "f"] \arrow[l, dashed, "\exists ! \,g"']\\
X & 
\end{tikzcd}
\]
\end{theorem}

\begin{proof}
Define $g \colon Y \to \CC(X)$ by mapping $y \in Y$ to $f[\up y]$. Since $Y$ is an Esakia root system, $\up y$ is a closed chain in $Y$ for every $y \in Y$. By \cref{lem:image of chain is chain}, $f[\up y] \in \CC(X)$ because $f$ is continuous and order preserving. So, $g$ is well defined.

We show that $g$ is continuous. By \cref{lem:facts CC(X):item5}, the clopen subsets of $\CC(X)$ of the form $\Box V$ and $\Diamond V$, with $V$ clopen in $X$, form a subbasis for the topology on $\CC(X)$. So, it is sufficient to show that $g^{-1}[\Box V]$ and $g^{-1}[\Diamond V]$ are clopen for every clopen subset $V$ of $X$. The definitions of $g$ and $\Diamond V$ imply that for every $y \in Y$ we have
\begin{align*}
y \in g^{-1}[\Diamond V] & \iff g(y) \in \Diamond V \iff f[\up y] \in \Diamond V \iff f[\up y] \cap V \neq \varnothing\\
& \iff \up y \cap f^{-1}[V] \neq \varnothing \iff y \in \down f^{-1}[V]. 
\end{align*}
Thus, 
\begin{equation}\label{eq:ginv diamond = down finv}
g^{-1}[\Diamond V] = \down f^{-1}[V].
\end{equation}
The continuity of $f$ yields that $f^{-1}[V]$ is clopen, and hence $g^{-1}[\Diamond V]=\down f^{-1}[V]$ is clopen because $Y$ is an Esakia space. We also have
\begin{align*}
g^{-1}[\Box V] &= g^{-1}[\CC(X) \setminus \Diamond(X \setminus V)] = Y \setminus g^{-1}[\Diamond(X \setminus V)]\\
&= Y \setminus \down f^{-1}[X \setminus V] = Y \setminus \down (Y \setminus f^{-1}[V]),
\end{align*}
where the first equality follows from \cref{lem:facts CC(X):item2}, the third from \cref{eq:ginv diamond = down finv}, and the remaining are consequences of the fact that preimages commute with complements. Since $V$ is clopen, $f$ is continuous, and $Y$ is an Esakia space, we obtain that $g^{-1}[\Box V]$ is clopen. This shows that $g$ is continuous.

To show that $g$ is a p-morphism, we first need to prove that
\begin{equation}\label{eq:fupy2 = upfy2 cap fupy1}
f[\up y_2] = \up f(y_2) \cap f[\up y_1], \text{ for every } y_1,y_2 \in Y \text{ with } y_1 \le y_2.
\end{equation}
From $y_1 \le y_2$ it follows that $\up y_2 \subseteq \up y_1$, and so $f[\up y_2] \subseteq \up f(y_2) \cap f[\up y_1]$ because $f$ is order preserving. If $x \in \up f(y_2) \cap f[\up y_1]$, then $f(y_2) \le x$ and there is $y_3 \in \up y_1$ such that $f(y_3) = x$. Since $\up y_1$ is a chain and $y_2,y_3 \in \up y_1$, we have $y_2 \le y_3$ or $y_3 \le y_2$. If $y_2 \le y_3$, then $y_3 \in \up y_2$, and so $x=f(y_3) \in f[\up y_2]$. If $y_3 \le y_2$, then $x=f(y_3) \le f(y_2) \le x$, and hence $x=f(y_2) \in f[\up y_2]$. In either case, $x \in f[\up y_2]$. Thus, \cref{eq:fupy2 = upfy2 cap fupy1} holds. It immediately follows that $y_1 \le y_2$ implies $f[\up y_1] \lec f[\up y_2]$, and hence $g(y_1) \lec g(y_2)$. Therefore, $g$ is order preserving. 
Let $y \in Y$ and $C \in \CC(X)$ be such that $g(y) \lec C$. Then $C=\up x \cap g(y)$ for some $x \in g(y) = f[\up y]$. Thus, $x=f(z)$ with $y \le z$, and so $C=\up f(z) \cap f[\up y]$. By \cref{eq:fupy2 = upfy2 cap fupy1}, $C= f[\up z] = g(z)$. This shows that $g$ is a p-morphism. 

Since $f$ is order preserving, $m g(y)=m(f[\up y])=f(y)$ for every $y \in Y$. Thus, $m \circ g=f$. It remains to show the uniqueness of $g$. Let $h \colon Y \to \CC(X)$ be a continuous p-morphism such that $m \circ h = f$. We show that $h(y)=g(y)$ for every $y \in Y$. 
By the definition of $m$, we have that $h(y)$ is a closed chain in $X$ whose last element is $f(y)$.
To prove $h(y) \subseteq g(y)$, consider $x \in h(y)$ and let $C = \up x \cap h(y) \in \CC(X)$. Then $h(y) \lec C$. Since $h$ is a p-morphism, there is $z \in \up y$ such that $C=h(z)$. Therefore, $x=m(C)=mh(z)=f(z)$, and so $x \in f[\up y]=g(y)$. Thus, $h(y) \subseteq g(y)$. To show that $g(y) \subseteq h(y)$, let $x \in g(y)$. By the definition of $g$, we have $x \in f[\up y]$, and so $x=f(z)$ with $y \le z$. Since $h$ is order preserving, $h(y) \lec h(z)$, which implies $h(z) \subseteq h(y)$. From $f=m \circ h$ it follows that $x=f(z)=mh(z)$. Thus, $x$ is the least element of $h(z)$, and hence $x \in h(z) \subseteq h(y)$. Thus, $g(y) \subseteq h(y)$. We have shown that $g=h$. Therefore, $g$ is the unique continuous p-morphism such that $m \circ g=f$.
\end{proof}

\begin{example}\label{ex:X and CCX}
Let $X=\mathbb{N} \cup \{ \infty\}$ be equipped with the topology generated by the subbasis $\{ \{n\}, X \setminus \{n\} \mid n \in \mathbb{N}\}$. Then $X$ is the one-point compactification of $\mathbb{N}$ with the discrete topology, and $X$ becomes a Priestley space once ordered as follows: $x_1 \le x_2$ iff $x_1=x_2$ or $x_1=\infty$.
\cref{fig:example X CCX} depicts $X$ and $\CC(X)$. All the points of $\CC(X)$ are isolated except for $\{ \infty\}$ and the topology on $\CC(X)$ is the one-point compactification of $\CC(X) \setminus \{ \{\infty\}\}$. Note that the map $m \colon \CC(X) \to X$ sends $\{x\}$ to $x$ for every $x \in X$ and sends $\{ \infty, n \}$ to $\infty$ for every $n \in \mathbb{N}$.
\begin{figure}[!h]
\begin{tikzpicture}
	\begin{pgfonlayer}{nodelayer}
		\node [style=black dot] (1) at (-11.25, 1) {};
		\node [style=black dot] (2) at (-10.25, 1) {};
		\node [style=black dot] (3) at (-9.25, 1) {};
		\node [style=black dot] (4) at (-6.75, -0.5) {};
		\node [style=small black dot] (5) at (-8.425, 1) {};
		\node [style=small black dot] (6) at (-8.025, 1) {};
		\node [style=small black dot] (7) at (-7.7, 1) {};
		\node [style=none] (8) at (-11.25, 1.5) {$0$};
		\node [style=none] (9) at (-10.25, 1.5) {$1$};
		\node [style=none] (10) at (-9.25, 1.5) {$2$};
		\node [style=none] (11) at (-6.7, -1) {\small $\infty$};
		\node [style=none] (12) at (-9.125, -1.75) {\small $X$};
		\node [style=black dot] (13) at (-2.45, 1) {};
		\node [style=black dot] (14) at (-1.2, 1) {};
		\node [style=black dot] (15) at (0.05, 1) {};
		\node [style=none] (17) at (-2.45, 1.5) {\small $\{0\}$};
		\node [style=none] (18) at (-1.2, 1.5) {\small $\{1\}$};
		\node [style=none] (19) at (0.05, 1.5) {\small $\{2\}$};
		\node [style=black dot] (20) at (-2.45, -0.5) {};
		\node [style=black dot] (21) at (-1.2, -0.5) {};
		\node [style=black dot] (22) at (0.05, -0.5) {};
		\node [style=small black dot] (23) at (1.125, 0.25) {};
		\node [style=small black dot] (24) at (1.55, 0.25) {};
		\node [style=small black dot] (25) at (1.925, 0.25) {};
		\node [style=black dot] (26) at (2.9, 0.25) {};
		\node [style=none] (27) at (2.9, 0.75) {\small $\{\infty\}$};
		\node [style=none] (28) at (-2.425, -1) {\small $\{\infty,0\}$};
		\node [style=none] (29) at (-1.175, -1) {\small $\{\infty,1\}$};
		\node [style=none] (30) at (0.075, -1) {\small $\{\infty,2\}$};
		\node [style=none] (31) at (-0.625, -1.75) {\small $\CC(X)$};
	\end{pgfonlayer}
	\begin{pgfonlayer}{edgelayer}
		\draw (1) to (4);
		\draw (4) to (2);
		\draw (3) to (4);
		\draw (13) to (20);
		\draw (14) to (21);
		\draw (15) to (22);
	\end{pgfonlayer}
\end{tikzpicture}

\caption{The Priestley space $X$ and the Esakia root system $\CC(X)$.}\label{fig:example X CCX}
\end{figure}

\end{example}

\begin{remark}\label{rem:adjunctions}
A straightforward argument using \cref{thm:univ prop CC(X)} shows that each $\Pries$-morphism $f \colon X_1 \to X_2$ yields an $\ERS$-morphism $\CC(f) \colon \CC(X_1) \to \CC(X_2)$ mapping $C \in \CC(X_1)$ to ${f[C] \in \CC(X_2)}$. It is immediate to verify that $\CC \colon \Pries \to \ERS$ is a functor, and \cite[Thm.~IV.1.2]{Mac71} implies that $\CC$ is right adjoint to the inclusion $\ERS \hookrightarrow \Pries$. 
\end{remark}

We are now ready to state the main result of the section, which provides a concrete dual description of the G\"odel algebra free over a given distributive lattice. In a nutshell, the following theorem states that if $L$ is a distributive lattice dual to a Priestley space $X$, then $\CC(X)$ is the Esakia root system dual to the G\"odel algebra free over $L$.
Recall that if $L$ is a distributive lattice and $X=L_*$ its dual Priestley space, then $\sigma_L \colon L \to X^*$ is the isomorphism that sends $a \in L$ to $\{ P \in X \mid a \in P \} \in X^*$. In particular, $\sigma_L(a)$ is a clopen upset of $X$ for every $a \in L$.

\begin{theorem}\label{thm:free Godel over L}
Let $L$ be a distributive lattice and $X=L_*$ its dual Priestley space. The G\"odel algebra $\CC(X)^*$ is free over $L$ via the map $e \colon L \to \CC(X)^*$ given by $e(a)=\Box \sigma_L(a)$. 
\end{theorem}

\begin{proof}
As we recalled in \cref{sec:prelim Priestley Esakia}, $(-)^* \colon \Pries \to \DL$ is a dual equivalence that restricts to a dual equivalence between $\ERS$ and $\HA$. It then follows from \cref{thm:univ prop CC(X)} that for every $H \in \GA$ and lattice homomorphism $f \colon X^* \to H$ there is a unique Heyting homomorphism $g \colon \CC(X)^* \to H$ such that $g \circ m^* = f$. Therefore, the G\"odel algebra $\CC(X)^*$ is free over the distributive lattice $X^*$ via the map $m^* \colon X^* \to \CC(X)^*$. Since $\sigma_L \colon L \to X^*$ is a lattice isomorphism, it is straightforward to verify that $\CC(X)^*$ is free over $L$ via the map $m^* \circ \sigma_L$. It remains to observe that $m^* \circ \sigma_L=e$. Since $\sigma_L(a)$ is a clopen upset of $X$ and $m^*$ is the inverse image under $m$,  \cref{lem:minv of up and down:item1} yields that $m^*(\sigma_L(a))=m^{-1}[\sigma_L(a)]=\Box \sigma_L(a) = e(a)$.
\end{proof}

\cref{thm:free Godel over L} directly generalizes the dual description of G\"odel algebras free over finite distributive lattices due to Aguzzoli, Gerla, and Marra by removing any restriction on the cardinality of the distributive lattices. Indeed, when $L$ is a finite distributive lattice, the statement of \cref{thm:free Godel over L} essentially coincides with \cite[Thm.~I]{AGM08}. In their setting, both $X$ and $\CC(X)$ are finite, so the topology does not play any role in their considerations, while it is fundamental when we deal with the infinite case.

We end the section with a dual description of free G\"odel algebras. 
Recall that a G\"odel algebra $G$ is said to be free over a set $S$ via a function $q \colon S \to G$ if for every G\"odel algebra $H$ and function $f \colon G \to H$ there is a unique Heyting homomorphism $g \colon G \to H$ such that $g \circ q = f$. A free distributive lattice over a set $S$ is defined similarly.
We will exploit the fact that the G\"odel algebra free over a set $S$ is isomorphic to the G\"odel algebra free over the distributive lattice that is free over $S$.
For this reason, we first recall the description of the Priestley spaces dual to free distributive lattices.

\begin{definition}
We denote by $\two$ the Priestley space consisting of the $2$-element chain $\{0 < 1\}$ with the discrete topology. For a set $S$, let $\two^S$ denote the set of all $S$-indexed sequences $(a_i)_{i \in S}$ of elements of $\two$. Then $\two^S$ becomes a Priestley space once equipped with the product topology and componentwise order; i.e., $(a_i)_{i \in S} \le (b_i)_{i \in S}$ iff $a_i \le b_i$ for every $i \in S$.
\end{definition}
Let $q \colon S \to (\two^S)^*$ be the map sending each $s \in S$ to the clopen upset $U_s \coloneqq \{(a_i)_{i \in S} \mid a_s=1\}$ of $\two^S$. The following fact is well known (see, e.g., \cite[Prop.~4.8]{GvG24}). 

\begin{proposition}\label{prop:free dist latt}
Let $S$ be a set. Then the distributive lattice $(\two^S)^*$ is free over $S$ via the map $q \colon S \to (\two^S)^*$.
\end{proposition}

The following theorem provides a dual description of free G\"odel algebras and  states that the G\"odel algebra free over a set $S$ is dual to the Esakia root system $\CC(\two^S)$.

\begin{theorem}\label{thm:free Godel over S}
Let $S$ be a set. Then the G\"odel algebra $\CC(\two^S)^*$ is free over $S$ via the map $r \colon S \to \CC(\two^S)^*$ given by $r(s)=\Box U_s$.
\end{theorem}

\begin{proof}
By \cref{prop:free dist latt}, the distributive lattice $(\two^S)^*$ is free over $S$ via the map $q$. Let $X=((\two^S)^*)_*$ be the double dual of $\two^S$. Then it is straightforward to verify that the G\"odel algebra $\CC(X)^*$ is free over $S$ via the map $e \circ q$, where $q \colon S \to (\two^S)^*$ and $e \colon (\two^S)^* \to \CC(X)^*$ are the maps appearing in \cref{prop:free dist latt,thm:free Godel over L}. To show that $\CC(\two^S)^*$ is free over $S$ via $r$, it is then sufficient to exhibit an isomorphism of G\"odel algebras $\varphi \colon \CC(X)^* \to \CC(\two^S)^*$ such that $\varphi \circ (e \circ q) = r$. Let $\varepsilon \colon \two^S \to X$ be the isomorphism of Priestley spaces described in \cref{sec:prelim Priestley Esakia}, where we omitted the subscript $2^S$ from $\varepsilon$ for ease of readability. Define $\varphi$ to be $\CC(\varepsilon)^*$. Since $\varepsilon$ is an isomorphism in $\Pries$, and $\CC$ and $(-)^*$ are functors, it follows that $\varphi$ is an isomorphism in $\ERS$. It then remains to show that $\varphi \circ (e \circ q) = r$. The definitions of $e$ and $q$ imply that $eq(s)=\Box \sigma q(s) = \Box \sigma (U_s)$ for every $s \in S$, where we omitted the subscript $(\two^S)^*$ from the isomorphism $\sigma_{(\two^S)^*} \colon (\two^S)^* \to X^*$. If $s \in S$, then
\begin{align*}
\varphi (eq(s)) & = \varphi(\Box \sigma q(s)) = \CC(\varepsilon)^{-1}(\Box \sigma q(s)) = \{C \in \CC(\two^S) \mid \varepsilon[C] \in \Box \sigma q(s) \}\\
& = \{C \in \CC(\two^S) \mid \varepsilon[C] \subseteq \sigma q(s) \} = \{C \in \CC(\two^S) \mid C \subseteq \varepsilon^{-1} [\sigma q(s)] \}\\
& = \{C \in \CC(\two^S) \mid C \subseteq \varepsilon^* \sigma q(s)] \} = \{C \in \CC(\two^S) \mid C \subseteq q(s) \}\\
& = \Box q(s) = r(s).
\end{align*}
In the above display, the first three, the sixth, and the last equalities follow from the definitions of $e$, $\varphi$, $\CC(\varepsilon)$, $\varepsilon^*$, and $r$, respectively. The fourth and the eighth equalities are a consequence of the definition of $\Box A$ for a subset $A$ of a Priestley space. The fifth equality is straightforward, and the seventh is an instance of the triangle identity $\varepsilon^* \circ \sigma=\text{id}_{(\two^S)^*}$, where $\text{id}_{(\two^S)^*}$ is the identity on $(\two^S)^*$ (see, e.g., \cite[Thm.~IV.1.1(ii)]{Mac71}). Therefore, $\varphi \circ (e \circ q) = r$, and this concludes the proof.
\end{proof}

\begin{remark}
Let $L$ be a distributive lattice. In general, it is extremely difficult to provide a tangible description of the Esakia dual $H_*$ of the Heyting algebra $H$ free over $L$. Our results allow to better understand a part of $H_*$. Indeed, since $\GA$ is a subvariety of $\HA$, the G\"odel algebra $G$ free over $L$ is isomorphic to a quotient of $H$. This dually correspond to the fact that the Esakia root system $G_*$ is isomorphic to the closed upset of $H_*$ given by $\{ y \in H_* \mid \up y \text{ is a chain}\}$. It then follows from \cref{thm:free Godel over L} that such a closed upset of $H_*$ is an Esakia root system isomorphic to $\CC(L_*)$.
Moreover, if $H$ is the Heyting algebra free over a set $S$, then it follows from \cref{thm:free Godel over S} that $\{ y \in H_* \mid \up y \text{ is a chain}\}$ is an Esakia root system isomorphic to $\CC(\two^S)$.
\end{remark}

We end the section with the analogues of \cref{thm:free Godel over L,thm:free Godel over S}, which provide a dual description of free algebras in $\GA_n$ for every $n \in \mathbb{N}$.

\begin{definition}
For $n \in \mathbb{N}$, we denote by $\CC_n(X)$ the subset of $\CC(X)$ consisting of all the points of $\CC(X)$ of depth less or equal to $n$. By \cref{prop:ERSn to ER right adjoint}, $\CC_n(X)$ is an Esakia root system of depth less or equal to $n$ with the order and topology induced by $\CC(X)$.
\end{definition}

The following proposition states that the elements of $\CC_n(X)$ are exactly all nonempty chains of $X$ of size (i.e., cardinality) at most $n$.

\begin{proposition}\label{prop:depth of chains}
Let $C$ be a nonempty subset of $X$. Then $C \in \CC_n(X)$ iff $C$ is a chain of size less or equal to $n$. 
\end{proposition}

\begin{proof} 
Note that, since every finite subset of a Stone space is closed, every finite chain of $X$ is closed in $X$.
\cref{lem:equivalent def lec} yields that $\upc C = \{\up x \cap C \mid x \in C\}$, which is a set in bijection with $C$. Thus, $C$ has depth $n$ iff it has size $n$.
\end{proof}

The following theorems show that if $L$ is a distributive lattice dual to $X=L_*$ and $S$ is a set, then $\CC_n(X)$ and $\CC_n(\mathsf{2}^S)$ are the $\GA_n$-algebras free over $L$ and $S$. Their proof is a straightforward adaptation of the proofs of \cref{thm:free Godel over L,thm:free Godel over S}.

\begin{theorem}\label{thm:free GAn over L}
Let $L$ be a distributive lattice and $X=L_*$. The $\GA_n$-algebra $\CC_n(X)^*$ is free over $L$ via the map $e_n \colon L \to \CC_n(X)^*$ given by $e_n(a)=\Box \sigma_L(a) \cap \CC_n(X)$. 
\end{theorem}

\begin{theorem}\label{thm:free GAn over set}
Let $S$ be a set. Then the $\GA_n$-algebra $\CC_n(\two^S)^*$ is free over $S$ via the map $r_n \colon S \to \CC_n(\two^S)^*$ given by $r_n(s)=\Box U_s \cap \CC_n(\two^S)$.
\end{theorem}

\section{Coproducts of G\"odel algebras}\label{sec:coproducts}

In \cite{DM06} the authors describe a procedure to compute the duals of binary coproducts of finite G\"odel algebras. 
In this section we utilize the machinery developed in \cref{sec:free Godel} to provide a dual description of arbitrary coproducts of G\"odel algebras without any restriction on the number of factors nor on the cardinality of the factors.
Our first goal is then to study products in the category $\ERS$ of Esakia root systems.
Products in the category of Esakia spaces are notoriously complicated. We show that products in $\ERS$ can be easily described in terms of collections of closed chains in the cartesian product of the factors.

We first recall the description of products in the category of Priestley spaces. Let $\{X_i \mid i \in I\}$ be a family of Priestley spaces and denote by $\prod_{i \in I} X_i$ their cartesian product equipped with the product topology and componentwise order. 
To simplify the notation, we will denote the product by $\timesp X_i$ when the set $I$ of indexes is clear from the context.
For each $i \in I$, we denote by $\pi_i \colon \timesp X_i \to X_i$ the projection onto $X_i$. The following proposition is well known and is an immediate consequence of the fact that products in the categories of topological spaces and posets coincide with cartesian products.

\begin{proposition}\label{prop:free Priestley}
Let $\{X_i \mid i \in I\}$ be a family of Priestley spaces. Then $\timesp X_i$ together with the maps $\pi_i \colon \timesp X_i \to X_i$ is the product of $\{X_i \mid i \in I\}$ in $\Pries$.
\end{proposition}

We now introduce the main construction of this section. Our first goal is to show that it gives the products in the category of Esakia root systems.

\begin{definition}
Let $\{Y_i \mid i \in I\}$ be a family of Esakia root systems. We define
\[
\bigotimes_{i \in I} Y_i \coloneqq \{ C \in \CC\big(\textstyle{\timesp} Y_i\big) \mid \pi_i[C]  \text{ is an upset of } Y_i \text{ for every } i \in I\}
\]
and equip it with the subspace topology and order induced by $\CC\big(\timesp Y_i\big)$.
\end{definition}

\begin{remark}
If $C \in \CC(\timesp Y_i)$, then  $\pi_i[C]$ is an upset of $Y_i$ iff it is a principal upset. Indeed, $C$ has a least element $m(C)$ by \cref{lem:CC in Pries min and max}, and so $\pi_i(m(C))$ is the least element of $\pi_i[C]$ because $\pi_i$ is an order preserving map. Therefore, if $\pi_i[C]$ is an upset, then $\pi_i[C]=\up \pi_i(m(C))$.
\end{remark}

\begin{example}
Let $\two = \{ 0 < 1\}$ be the $2$-element chain with the discrete topology. Consider $C_1,C_2 \in \CC(\two \times \two)$ defined as follows
\[
C_1= \{(1,0),(1,1)\} \qquad \text{and} \qquad C_2= \{(0,0),(0,1)\}.
\]
The white dots in \cref{fig:C1 and C2 projections} represent $C_1$ with its projections $\pi_1[C_1], \pi_2[C_1] \in \CC(\two)$ on the left and $C_2$ with its projections $\pi_1[C_2],\pi_2[C_2] \in \CC(\two)$ on the right. 
Since both $\pi_1[C_1]=\{1\}$ and $\pi_2[C_1]=\{0,1\}$ are upsets of $\two$, we have that $C_1 \in \two \otimes \two$. However, $C_2 \notin \two \otimes \two$ because $\pi_1[C_2]=\{0\}$ is not an upset of $\two$.
\begin{figure}[!h]
\begin{tikzpicture}
	\begin{pgfonlayer}{nodelayer}
		\node [style=white dot] (0) at (-4, 1) {};
		\node [style=white dot] (1) at (-5, 0) {};
		\node [style=black dot] (2) at (-4, -1) {};
		\node [style=black dot] (4) at (-3, 0) {};
		\node [style=black dot] (5) at (4, 1) {};
		\node [style=black dot] (6) at (3, 0) {};
		\node [style=white dot] (7) at (4, -1) {};
		\node [style=white dot] (9) at (5, 0) {};
		\node [style=white dot] (10) at (-6.125, -1.125) {};
		\node [style=black dot] (11) at (-5.125, -2.125) {};
		\node [style=black dot] (12) at (1.875, -1.125) {};
		\node [style=white dot] (13) at (2.875, -2.125) {};
		\node [style=white dot] (14) at (-2.875, -2.125) {};
		\node [style=white dot] (16) at (-1.875, -1.125) {};
		\node [style=white dot] (17) at (5.125, -2.125) {};
		\node [style=white dot] (19) at (6.125, -1.125) {};
		\node [style=none] (20) at (-4.75, -0.75) {};
		\node [style=none] (22) at (-5.375, -1.375) {};
		\node [style=none] (23) at (-3.25, -0.75) {};
		\node [style=none] (24) at (-2.625, -1.375) {};
		\node [style=none] (25) at (3.25, -0.75) {};
		\node [style=none] (26) at (2.625, -1.375) {};
		\node [style=none] (27) at (4.75, -0.75) {};
		\node [style=none] (28) at (5.375, -1.375) {};
		\node [style=none] (32) at (5.375, -0.75) {\small $\pi_2$};
		\node [style=none] (33) at (2.625, -0.75) {\small $\pi_1$};
		\node [style=none] (34) at (-2.625, -0.75) {\small $\pi_2$};
		\node [style=none] (35) at (-5.375, -0.75) {\small $\pi_1$};
	\end{pgfonlayer}
	\begin{pgfonlayer}{edgelayer}
		\draw (0) to (4);
		\draw (4) to (2);
		\draw (1) to (2);
		\draw (1) to (0);
		\draw (5) to (9);
		\draw (9) to (7);
		\draw (6) to (7);
		\draw (6) to (5);
		\draw (10) to (11);
		\draw (12) to (13);
		\draw (16) to (14);
		\draw (19) to (17);
		\draw [style=to] (20.center) to (22.center);
		\draw [style=to] (23.center) to (24.center);
		\draw [style=to] (25.center) to (26.center);
		\draw [style=to] (27.center) to (28.center);
	\end{pgfonlayer}
\end{tikzpicture}

\caption{The chains $C_1$ and $C_2$ of $\two \times \two$ and their projections.}\label{fig:C1 and C2 projections}
\end{figure}
\end{example}

We will show that $\timesg Y_i$ is the product of $\{Y_i \mid i \in I\}$ in the category $\ERS$. We begin by proving that $\timesg Y_i$ is an Esakia root system, but we first need to recall the following technical fact.

\begin{lemma}{\cite[Prop.~2.6(iv)]{Pri84}}\label{prop:separation closed upsets}
Let $A,B \subseteq X$ be closed subsets of a Priestley space such that $\up A \cap \down B = \varnothing$. Then there is a clopen upset $U$ and a clopen downset $D$ such that $U \cap D = \varnothing$,  $A \subseteq U$, and $B \subseteq D$.
\end{lemma}

\begin{theorem}\label{thm:prod ERS}
$\timesg Y_i$ is an Esakia root system.
\end{theorem}

\begin{proof}
Since $\timesp Y_i$ is a Priestley space, \cref{thm:CC in ERS} yields that $\CC(\timesp Y_i)$ is an Esakia root system. Since closed upsets of Esakia spaces equipped with the subspace topology and the restriction of the order are Esakia spaces (see, e.g., \cite[Lem.~3.4.11]{Esa19}), it is sufficient to show that $\timesg Y_i$ is a closed upset of $\CC(\timesp Y_i)$. 

To prove that $\timesg Y_i$ is an upset of $\CC(\timesp Y_i)$, let $C_1 \in \timesg Y_i$ and $C_2 \in \CC(\timesp Y_i)$ such that $C_1 \lec C_2$. We show that $C_2 \in \timesg Y_i$, which means that $\up \pi_i[C_2] = \pi_i[C_2]$ for every $i \in I$. Let $i \in I$ and $y \in \up \pi_i[C_2]$. Since $C_2 \subseteq C_1$ and $\pi_i[C_1]$ is an upset of $Y_i$, we obtain that $y \in \up \pi_i[C_2] \subseteq \up \pi_i[C_1] = \pi_i[C_1]$. So, there is $x \in C_1$ such that $\pi_i(x)=y$. Let $m(C_2)$ be the least element of $C_2$. Since $x,m(C_2) \in C_1$ and $C_1$ is a chain, we have that $x \le m(C_2)$ or $m(C_2) \le x$. If $x \le m(C_2)$, then $y=\pi_i(x) \le \pi_i(m(C_2))$. We have that $\pi_i(m(C_2)) \le y$ because $y \in \up \pi_i[C_2]$ and $\pi_i$ is order preserving. So, $y=\pi_i(m(C_2)) \in \pi_i[C_2]$. If $m(C_2) \le x$, then $x \in C_2$ because $C_1 \lec C_2$. Thus, $y=\pi_i(x) \in \pi_i[C_2]$. In either case, $y \in \pi_i[C_2]$. This shows that $\pi_i[C_2]$ is an upset of $Y_i$ for each $i \in I$, and hence $C_2 \in \CC(\timesp Y_i)$.

It remains to show that $\timesg Y_i$ is a closed subset of $\CC(\timesp Y_i)$. Let $C \in \CC(\timesp Y_i)$ be such that $C \notin \timesg Y_i$. We show that $C$ is contained in an open subset of $\CC(\timesp Y_i)$ disjoint from $\timesg Y_i$. 
Since $C \notin \timesg Y_i$, there is $i \in I$ such that $\pi_i[C]$ is not an upset of $Y_i$. Then $\pi_i[C] \ne \up \pi_i[C]$. 
Because $\pi_i[C] \ne \up \pi_i[C]$, there is $x \in \up \pi_i[C]$ that is not in $\pi_i[C]$.
From $x \notin \pi_i[C]$ it follows that $\up x \cap (\down x \cap \pi_i[C]) = \varnothing$ and $\down x \cap (\up x \cap \pi_i[C])=\varnothing$. Since $\down x \cap \pi_i[C]$ and $\up x \cap \pi_i[C]$ are closed subsets of $Y_i$, two applications of \cref{prop:separation closed upsets} yield a clopen upset $U$ disjoint from $\down x \cap \pi_i[C]$ such that $x \in U$ and a clopen downset $D$ disjoint from $\up x \cap \pi_i[C]$ such that $x \in D$. 
Then
\begin{equation}\label{eq:inclusions down up x cap piiC}
\down x \cap \pi_i[C] \subseteq \down (U \cap D), \quad \down x \cap \pi_i[C] \subseteq Y_i \setminus (U \cap D), \quad \text{and} \quad \up x \cap \pi_i[C] \subseteq U \setminus D,
\end{equation}
where the first inclusion holds because $x \in U \cap D$, the second because $\down x \cap \pi_i[C] \subseteq Y_i \setminus U$, and the third follows from $x \in U$ and $\up x \cap \pi_i[C] \subseteq Y_i \setminus D$. Since $\up \pi_i[C] \subseteq \up \pi_i(m(C))$ and $Y_i$ is a root system, we have that $\up \pi_i[C]$ is a chain. It then follows from $x \in \up \pi_i[C]$ and $\pi_i[C] \subseteq \up \pi_i[C]$ that $\pi_i[C]=(\down x \cap \pi_i[C]) \cup (\up x \cap \pi_i[C])$. The inclusions in \eqref{eq:inclusions down up x cap piiC} imply
that
\begin{equation}\label{eq:proof prod in box}
\pi_i[C] = (\down x \cap \pi_i[C]) \cup (\up x \cap \pi_i[C]) \subseteq (\down(U \cap D)  \setminus  (U \cap D)) \cup (U  \setminus  D).
\end{equation}
Since $x \in \up \pi_i[C]$, we have that 
$\pi_i(m(C)) \le x$, and hence $\pi_i(m(C)) \in \down x \cap \pi_i[C]$. It then follows from the inclusions in \eqref{eq:inclusions down up x cap piiC} that $\pi_i(m(C)) \in \down(U \cap D)$ and $\pi_i(m(C)) \notin U \cap D$. Thus, 
\begin{equation}\label{eq:proof prod in diamond}
\pi_i[C] \cap (\down(U \cap D)  \setminus  (U \cap D)) \neq \varnothing.
\end{equation}
Then \eqref{eq:proof prod in box} and \eqref{eq:proof prod in diamond} imply
\[
C \in \mathcal{U} \coloneqq \Box \pi_i^{-1}[(\down(U \cap D)  \setminus  (U \cap D)) \cup (U  \setminus  D)] \cap \Diamond \pi_i^{-1}[\down(U \cap D)  \setminus  (U \cap D)].
\]
We show that $\mathcal{U}$ is a clopen of $\CC(\timesp Y_i)$ disjoint from $\timesg Y_i$. 
Since $Y_i$ is an Esakia space, $U$ and $D$ are clopen, and $\pi_i$ is continuous, we have that $\mathcal{U}$ is clopen. To prove that each $K \in \mathcal{U}$ is not in $\timesg Y_i$, assume that $K \in \mathcal{U} \cap \timesg Y_i$. Then $K \in \mathcal{U} \subseteq \Diamond \pi_i^{-1}[\down(U \cap D)  \setminus  (U \cap D)]$, and hence there is $y \in K$ such that $\pi_i(y) \in \down (U \cap D)  \setminus  (U \cap D)$. This implies that there is $z \in Y_i$ such that 
\[
z \in \up \pi_i(y) \cap U \cap D \subseteq \up \pi_i (m(K)) \cap U \cap D \subseteq \pi_i[K] \cap U \cap D,
\]
where $\up \pi_i (m (K)) \subseteq \up \pi_i[K] = \pi_i[K]$ because $K \in \timesg Y_i$.
However, 
\[
K \in \mathcal{U} \subseteq \Box \pi_i^{-1}[(\down(U \cap D)  \setminus  (U \cap D)) \cup (U  \setminus  D)] \subseteq \Box \pi_i^{-1}[Y_i \setminus (U \cap D)],
\] 
and so $\pi_i[K] \subseteq Y_i \setminus (U \cap D)$. This contradicts the existence of $z \in \pi_i[K] \cap U \cap D$. Therefore, $\mathcal{U}$ is a clopen subset of $\CC(\timesp Y_i)$ containing $C$ that is disjoint from $\timesg Y_i$. This shows that $\timesg Y_i$ is closed in $\CC(\timesp Y_i)$, and concludes the proof that $\timesg Y_i$ is an Esakia root system. 
\end{proof}

We are now ready to show that $\timesg Y_i$ is the product in $\ERS$ of the family $\{ Y_i \mid i \in I\}$. For each $i \in I$, we let $p_i \colon \timesg Y_i \to Y_i$ be the map that sends $C \in \timesg Y_i$ to $\pi_i(m(C)) \in Y_i$, where $m(C)$ is the least element of $C$.

\begin{theorem}\label{thm:univ prop prod in ERS}
Let $\{Y_i \mid i \in I \}$ be a family of Esakia root systems. Then $\timesg Y_i$ together with the maps $p_i \colon \timesg Y_i \to Y_i$ is the product of $\{Y_i \mid i \in I \}$ in $\ERS$.
\end{theorem}

\begin{proof}
We first show that $p_i \colon \timesg Y_i \to Y_i$ is a continuous p-morphism for every $i \in I$.
We have that $p_i = \pi_i \circ m$, and so it is a continuous map because every projection $\pi_i$ is continuous and $m$ is continuous by  \cref{lem:m Priestley map}.
We show that $p_i$ is a p-morphism. Both $\pi_i$ and $m$ are order preserving, so $p_i$ is also order preserving. Let $C \in \timesg Y_i$ and $y \in Y_i$ such that $p_i(C) \le y$. Then $\pi_i(m(C)) \le y$, and hence $y \in \up \pi_i(m(C)) \subseteq \up \pi_i[C]$. Since $C \in \timesg Y_i$, we have that $\up \pi_i[C]=\pi_i[C]$. Thus, $y \in \pi_i[C]$. So, there is $x \in C$ such that $\pi_i(x)=y$. Let $K = C \cap \up x$. Then $C \lec K$, and hence $K \in \timesg Y_i$ because $\timesg Y_i$ is an upset in $\CC(\timesp Y_i)$ as shown in the proof of \cref{thm:prod ERS}. Therefore, $y=\pi_i(x)=\pi_i(m(K)) = p_i(K)$. This shows that $p_i$ is a p-morphism.

It remains to verify the universal property of products. Let $Z$ be an Esakia root system and $f_i \colon Z \to Y_i$ a continuous p-morphism for each $i \in I$. 
By \cref{prop:free Priestley}, there is a map $q \colon Z \to \timesp Y_i$ sending $z \in Z$ to $(f_i(z))_{i \in I}$ that is continuous and order preserving. Thus, \cref{thm:univ prop CC(X)} yields a continuous p-morphism $g \colon Z \to \CC(\timesp Y_i)$ given by $g(z) = q[\up z]= \{ (f_i(w))_{i \in I} \mid w \in \up z \}$ for every $z \in Z$. Since each $f_i$ is a p-morphism, we have that $f_i[\up z]$ is an upset of $Y_i$ for every $z \in Z$.
So, $\up \pi_i[g(z)] = \up f_i[\up z]=f_i[\up z] = \pi_i[g(z)]$. Therefore, $g(z) \in \timesg Y_i$, and hence $g$ restricts to a continuous p-morphism $g \colon Z \to \timesg Y_i$. Moreover, $p_i g(z)=\pi_i m g(z)=\pi_i ((f_i(z))_{i \in I})=f_i(z)$ for every $z \in Z$. Thus, $p_i \circ g = f_i$ for each $i \in I$.
We now show that $g$ is the unique map with such properties. Suppose that $h \colon Z \to \timesg Y_i$ is a continuous p-morphism such that $p_i \circ h = f_i$ for every $i \in I$. Since $\timesg Y_i$ is an upset of $\CC(\timesp Y_i)$, it follows that $h \colon Z \to \CC(\timesp Y_i)$ is a continuous p-morphism. We also have that $\pi_i m(h(z)) = p_i(h(z)) = f_i(z)$ for all $i \in I$ and $z \in Z$. Thus, $m(h(z))=(f_i(z))_{i \in I} = m(g(z))$ for each $z \in Z$. Then $h \colon Z \to \CC(\timesp Y_i)$ is a continuous p-morphism such that  $m \circ h=m \circ g$. It follows from \cref{thm:univ prop CC(X)} that $h=g$.
\end{proof}

We are now ready to describe the Esakia duals of coproducts of G\"odel algebras.

\begin{theorem}\label{thm:coproduct Godel}
Let $\{ G_i \mid i \in I\}$ be a family of G\"odel algebras and $Y_i=(G_i)_*$ their Esakia duals. Then the coproduct of $\{ G_i \mid i \in I\}$ in $\GA$ is given by $(\timesg Y_i)^*$ together with the maps sending $a \in G_i$ to $\Box \pi_i^{-1}[\sigma_{G_i}(a)] \in (\timesg Y_i)^*$ for each $i \in I$.
\end{theorem}

\begin{proof}
Since $(-)^* \colon \ERS \to \GA$ is a dual equivalence of categories, it sends products into coproducts. Thus, \cref{thm:univ prop prod in ERS} yields that $(\timesg Y_i)^*$ together with the maps $p_i^*$ for each $i \in I$ is the coproduct of $\{ Y_i^* \mid i \in I\}$ in $\GA$. Since $\sigma_{G_i} \colon G_i \to Y_i^*$ is an isomorphism of G\"odel algebras for each $i \in I$, it is straightforward to verify that $(\timesg Y_i)^*$ together with the maps $p_i^* \circ \sigma_{G_i}$ is the coproduct of $\{ G_i \mid i \in I\}$ in $\GA$. It remains to observe that $p_i^* \circ \sigma_{G_i}$ maps $a \in G_i$ to $\Box \pi_i^{-1}[\sigma_{G_i}(a)] \in (\timesg Y_i)^*$. Since $p_i=\pi_i \circ m$, we obtain that $p_i^*(\sigma_{G_i}(a))=m^{-1}[\pi_i^{-1}[\sigma_{G_i}(a)]]$. Because $\sigma_{G_i}(a)$ is a clopen upset of $Y_i$, \cref{lem:minv of up and down:item1} yields that $p_i^*(\sigma_{G_i}(a))=\Box\pi_i^{-1}[\sigma_{G_i}(a)]$.
\end{proof}

For a G\"odel algebra $G$, let $d(G)$ be the least $n \in \mathbb{N}$ such that $G \in \GA_n$, and if there is no such $n \in \mathbb{N}$, let $d(G)=\infty$. We call $d(G)$ the \emph{depth} of $G$. It is an immediate consequence of Esakia duality for $\GA_n$ that $d(G)=d(G_*)$, where $d(G_*)$ is the depth of the Esakia root system $G_*$ as defined in \cref{sec:prelim Priestley Esakia}.
Thanks to the dual description of coproducts we just obtained, we have a way to compute the depth of coproducts of G\"odel algebras. 

\begin{theorem}\label{thm:depth coproducts}
Let $\{G_i \mid i \in I\}$ be a family of nontrivial G\"odel algebras. Then the coproduct of $\{G_i \mid i \in I\}$ has depth $1 + \sum_{i \in I} (d(G_i)-1)$, where  $d(G_i) \in \mathbb{N} \cup \{ \infty\}$ is the depth of $G_i$ for each $i \in I$.\footnote{We mean that the expression $1+\sum_{i \in I} (d(G_i)-1)$ equals $\infty$ when $J\coloneqq\{i \in I \mid d(G_i) > 1\}$ is infinite or $d(G_i) = \infty$ for some $i \in I$, and that $1+\sum_{i \in I} (d(G_i)-1)=1+\sum_{i \in J} (d(G_i)-1)$, otherwise.}
\end{theorem}

\begin{proof}
Let $Y_i=(G_i)_*$ be the Esakia root system dual to $G_i$. \cref{thm:coproduct Godel} implies that the depth of the coproduct of $\{G_i \mid i \in I\}$ coincides with the depth of $\timesg Y_i$. Since $d(G_i)=d(Y_i)$ for every $i \in I$, it is sufficient to show that $d(\timesg Y_i)=1 + \sum_{i \in I} (d(Y_i)-1)$.
We first prove the following technical fact.

\begin{claim}\label{claim:chains in products}
Let $w_i \in \max Y_i$ for every $i \in I$. Let also $i_1, \dots, i_n \in I$ and $z_{i_j} \in Y_{i_j}$ such that $z_{i_j} \le w_{i_j}$ for each $j=1, \dots, n$.
Define $C_j \subseteq \timesp Y_i$ for each $j=1, \dots, n$ as follows
\begin{align*}
(y_i)_{i \in I} \in C_j \iff \begin{cases}
y_i=z_i &\text{if } i \in \{i_1, \dots, i_{j-1}\},\\
y_i \in \up z_{i_j} &\text{if } i=i_j,\\
y_i=w_i &\text{if } i \notin \{i_1, \dots, i_j\}.
\end{cases}
\end{align*}
Then $C \coloneqq C_1 \cup \dots \cup C_n$ is a point of depth $1 + \sum_{j=1}^n (d(z_{i_j})-1)$ in $\timesg Y_i$, where $d(z_{i_j})$ is the depth of $z_{i_j}$ in $Y_{i_j}$.
\end{claim}

\begin{proofclaim}
For every $j$, we have
\[
C_j=\bigcap \{ \pi_i^{-1}[z_i] \mid  i \in \{i_1, \dots, i_{j-1}\}\} \cap \pi_{i_j}^{-1}[\up z_{i_j}] \cap \bigcap \{ \pi_i^{-1}[w_i] \mid  i \notin \{i_1, \dots, i_j\}\}, 
\]
and hence $C_j$ is a closed subset of $\timesp Y_i$ because $\{ z_i\}$, $\up z_{i_j}$, and $\{w_i\}$ are all closed and $\pi_i$ is continuous for every $i \in I$. 
Since $Y_{i_j}$ is a root system, $\up z_{i_j}$ is a chain, and hence $C_j \in \CC(\timesp Y_i)$. 
If $1 \le j < n$, then the least element of $C_j$ is the greatest element of $C_{j+1}$. 
Therefore, $C = C_1 \cup \dots \cup C_n$ is a chain in $\timesp Y_i$. Since $C_1, \dots, C_n$ are closed in $\timesp Y_i$, we obtain that $C \in \CC(\timesp Y_i)$. By the definition of $C_1, \dots, C_n$, it follows that $\pi_i[C]=\up z_i$ if $i \in \{i_1, \dots, i_n\}$, and $\pi_i[C]=\{w_i\}$, otherwise. Therefore, for every $i \in I$ the set $\pi_i[C]$ is an upset of $Y_i$, and hence $C \in \timesg Y_i$. By \cref{prop:depth of chains}, the depth of $C$ in $\bigotimes_i Y_i$ coincides with its size. If $d(z_{i_j}) = \infty$ for some $j$, then $C$ is infinite, and so has depth $\infty=1 + \sum_{j=1}^n (d(z_{i_j})-1)$ in $\timesg Y_i$.  Otherwise, $C$ has size $1+\sum_{j=1}^n (d(z_{i_j})-1)$. Indeed, each $C_j$ has size $d(z_{i_j})$ and the least element of $C_j$ coincides with the greatest element of $C_{j+1}$ for every $j=1, \dots, n-1$. This shows that $C$ is an element of $\timesg Y_i$ of depth $1 + \sum_{j=1}^n (d(z_{i_j})-1)$.
\end{proofclaim}

We first consider the case in which there is $k \in I$ with $d(Y_k)=\infty$. Then there is an element of $Y_k$ of infinite depth or there are elements of $Y_k$ of arbitrary large finite depth. Suppose first that there is $z_k \in Y_k$ such that $d(z_k)=\infty$. Since every $G_i$ is nontrivial, all the $Y_i$ are nonempty. Then $\max Y_i$ and $\max (\up z_k)$ are nonempty by~\cite[Prop.~2.6]{Pri84}. So, we can pick $w_i \in \max Y_i$ for every $i \in I$ so that $z_k \le w_k$. By \cref{claim:chains in products}, there is $C \in \timesg Y_i$ of infinite depth. Thus, $d(\timesg Y_i)=\infty$. Suppose now that for each $n \in \mathbb{N}$ there is $z_k \in Y_k$ such that $d(z_k) \ge n$. By arguing as in the previous case, we obtain that there are elements of $\timesg Y_i$ of arbitrarily large finite depth, and hence $d(\timesg Y_i)=\infty$. So, we can assume that $d(Y_i) \ne \infty$ for every $i \in I$. Suppose there are infinitely many $i \in I$ such that $d(Y_i)>1$. Then for each $n \in \mathbb{N}$ we can find $i_1, \dots, i_n \in I$ and $z_{i_j} \in Y_{i_j}$ such that $d(z_{i_j}) \ge 2$. \cref{claim:chains in products} allows us to construct $C \in \timesg Y_i$ of depth $1 + \sum_{j=1}^n (d(z_{i_j})-1) > n$. It follows that $d(\timesg Y_i)=\infty$. The last case to consider is when $d(Y_i) \ne \infty$ for every $i \in I$ and $\{i \in I \mid d(Y_i) > 1\}$ is finite. Let $\{i \in I \mid d(Y_i) > 1\} = \{ i_1, \dots, i_n\}$ and $z_{i_j} \in Y_{i_j}$ such that $d(z_{i_j})=d(Y_{i_j})$. Then \cref{claim:chains in products} implies that there is $C \in  \timesg Y_i$ of depth $1 + \sum_{j=1}^n (d(z_{i_j})-1)$, and hence
\[
d(\textstyle{\timesg} Y_i) \ge 1 + \sum_{j=1}^n (d(z_{i_j})-1) = 1 + \sum_{j=1}^n (d(Y_{i_j})-1).
\] 
We now show that $d(\timesg Y_i) \le 1 + \sum_{j=1}^n (d(Y_{i_j})-1)$. Let $K \in \timesg Y_i$. We prove that the size of $K$ is smaller or equal to $1+\sum_{j=1}^n (d(Y_{i_j})-1)$. If $i \notin \{i_1, \dots, i_n \}$, then $\pi_i[K]=\{w_i\}$ with $w_i \in \max Y_i$ because $d(Y_i)=1$. If $j=1, \dots, n$, let $\pi_{i_j}[K]=\up z_{i_j}$ for some $z_{i_j} \in Y_{i_j}$. Thus, every $\pi_i[K]$ is finite, and it is a singleton for all but finitely many $i \in I$. Then $K$ is a finite chain because $K \subseteq \prod_i \pi_i[K]$. 
Let 
\[
W=\{ (j,y) \mid j \in \{1, \dots, n\} \text{ and } y \in \up z_{i_j} \setminus \{z_{i_j}\}\}.
\]
For every $(j,y) \in W$ we have that $\pi_{i_j}^{-1}[y] \cap K$ is a finite nonempty chain because $K$ is a finite chain and $y \in \up z_{i_j} = \pi_{i_j}[K]$. Recall that $m(K)$ denotes the least element of $K$. Define $f \colon W \to K \setminus m(K)$ by mapping $(j,y) \in W$ to the least element of $\pi_{i_j}^{-1}[y] \cap K$, which belongs to $K \setminus m(K)$ because $\pi_{i_j}(m(K))=z_{i_j}$ and $y \neq z_{i_j}$. We show that $f$ is onto. Let $\overline{y}=(y_i)_{i \in I} \in K \setminus m(K)$. Then there exists $j \in \{1, \dots, n\}$ such that the predecessor of $\overline{y}$ in $K$ differs from $\overline{y}$ in the $i_j$-th component. So, $y_{i_j} \neq z_{i_j}$, and hence $(j,y_{i_j}) \in W$. From the definition of $f$ it follows that $f(j,y_{i_j}) = \overline{y}$. Thus, $f \colon W \to K \setminus m(K)$ is onto. Then the size of $K \setminus m(K)$ is smaller or equal to the cardinality of $W$, which is $\sum_{j=1}^n (d(z_{i_j})-1)$.
Since $d(z_{i_j}) \le d(Y_{i_j})$ for every $j$, we obtain that the size of $K$ is smaller or equal to
$1 + \sum_{j=1}^n (d(Y_{i_j})-1)$ for every $K \in \timesg Y_i$.
By \cref{prop:depth of chains}, the size of $K$ coincides with its depth in $\timesg Y_i$, so we get $d(\timesg Y_i) \le 1 + \sum_{j=1}^n (d(Y_{i_j})-1)$. We assumed that $d(Y_i) \ne \infty$ for every $i \in I$ and that $\{i \in I \mid d(Y_i) > 1\}$ is finite, so $1 + \sum_{j=1}^n (d(Y_{i_j})-1) = 1 + \sum_{i \in I} (d(Y_i)-1)$. It follows that $d(\timesg Y_i) \le 1 + \sum_{i \in I} (d(Y_i)-1)$. This concludes the proof that $d(\timesg Y_i)=1+\sum_{i \in I} (d(Y_i)-1)$.  By what we observed at the beginning of the proof, it follows that the coproduct of $\{G_i \mid i \in I\}$ has depth $1 + \sum_{i \in I} (d(G_i)-1)$.
\end{proof}

We end this section with the dual description of coproducts in $\GA_n$. If $\{Y_i \mid i \in I\}$ is a family of Esakia root systems, we denote by $(\timesg Y_i)_n$ the set of elements of $\timesg Y_i$ of depth less or equal to $n$ equipped with the subspace topology and order induced by $\timesg Y_i$. It is straightforward to see that 
\[
(\textstyle{\timesg} Y_i)_n = \{ C \in \CC_n\big(\textstyle{\timesp} Y_i\big) \mid \pi_i[C]  \text{ is an upset of } Y_i \text{ for every } i \in I\}.
\]

The following theorems are immediate consequences of \cref{cor:subn preserving limits,thm:Esakia duality GAn,thm:univ prop prod in ERS}.

\begin{theorem}
Let $n \in \mathbb{N}$ and $\{Y_i \mid i \in I\} \subseteq \ERS_n$. Then the product of $\{Y_i \mid i \in I\}$ in $\ERS_n$ is given by $(\timesg Y_i)_n$ together with the maps $p_i \colon (\timesg Y_i)_n \to Y_i$ sending $C \in (\timesg Y_i)_n$ to $\pi_i(m(C)) \in Y_i$.
\end{theorem}

\begin{theorem}\label{thm:coproduct GAn}
Let $\{ G_i \mid i \in I\}$ be a family of $\GA_n$-algebras and $Y_i=(G_i)_*$ their Esakia duals. Then the coproduct of $\{ G_i \mid i \in I\}$ in $\GA_n$ is given by $((\timesg Y_i)_n)^*$ together with the maps sending $a \in G_i$ to $\Box \pi_i^{-1}[\sigma_{G_i}(a)] \cap (\timesg Y_i)_n \in ((\timesg Y_i)_n)^*$ for each $i \in I$. 
\end{theorem}

Since $\GA_n$ is a subvariety of $\GA$, the inclusion $\GA_n \hookrightarrow \GA$ is a right adjoint (see, e.g., \cite[Cor.~9.4.15]{Ber15}), and hence all limits coincide in $\GA_n$ and $\GA$. This is not true for colimits. The following corollary characterizes when coproducts in $\GA_n$ and $\GA$ coincide.

\begin{corollary}
Let $\{ G_i \mid i \in I\}$ be a family of nontrivial $\GA_n$-algebras with $n \ge 1$. Then the coproducts of $\{ G_i \mid i \in I\}$ in $\GA$ and $\GA_n$ are isomorphic iff $\sum_{i \in I} (d(G_i)-1) \le n-1$.
\end{corollary}

\begin{proof}
Let $Y_i=(G_i)_*$ be the Esakia root system dual to $G_i$ for each $i \in I$. By \cref{thm:coproduct Godel}, the coproduct of $\{ G_i \mid i \in I\}$ in $\GA$ is dual to $\timesg Y_i$, while \cref{thm:coproduct GAn} yields that the coproduct of $\{ G_i \mid i \in I\}$ in $\GA_n$ is dual to $(\timesg Y_i)_n$. Thus, the two coproducts are isomorphic iff $\timesg Y_i$ and $(\timesg Y_i)_n$ are isomorphic Esakia root systems, which happens exactly when $d(\timesg Y_i) \le n$. \cref{thm:depth coproducts} yields that $d(\timesg Y_i)=1 + \sum_{i \in I} (d(G_i)-1)$. Therefore, the two coproducts coincide when $1 + \sum_{i \in I} (d(G_i)-1) \le n$, which is equivalent to $\sum_{i \in I} (d(G_i)-1) \le n-1$.
\end{proof}

\section{Free G\"odel algebras as bi-Heyting algebras}\label{sec:bi-Heyting}
 
A distributive lattice is called a \emph{co-Heyting algebra} when its order dual is a Heyting algebra, and a \emph{bi-Heyting algebra} is a Heyting algebra that is also a co-Heyting algebra. It is shown in~\cite{Ghi92} that Heyting algebras free over finite distributive lattices are always bi-Heyting algebras. This result implies the surprising fact that all free Heyting algebras with finitely many free generators are bi-Heyting algebras. 

In this section we provide a necessary and sufficient condition for G\"odel algebras free over distributive lattices to be bi-Heyting algebras, and obtain as a consequence that any free G\"odel algebra is a bi-Heyting algebra. We end the section by showing that the situation is very different in the bounded depth setting, as every free $\GA_n$-algebra with infinitely many free generators is not a bi-Heyting algebra. 

We say that a Priestley space $X$ is a \emph{co-Esakia space} when $\up V$ is clopen for every $V$ clopen subset of $X$, and that $X$ is a \emph{bi-Esakia space} when it is at the same time an Esakia and a co-Esakia space. 
Observe that the Priestley space $X$ described in
\cref{ex:X and CCX} is a co-Esakia space that is not an Esakia space, while $\CC(X)$ from the same example is a bi-Esakia space.
There are analogues of Esakia duality for co-Heyting and bi-Heyting algebras. 
For our purposes, we only need the following well-known proposition; we sketch its proof due to the lack of a reference.

\begin{proposition}\plabel{prop:Esakia duality for cobiHeyting}
Let $L$ be a distributive lattice and $X=L_*$ its dual Priestley space.
\begin{enumerate}
\item $L$ is a co-Heyting algebra iff $X$ is a co-Esakia space.
\item $L$ is a bi-Heyting algebra iff $X$ is a bi-Esakia space.
\end{enumerate}
\end{proposition}

\begin{proof}
Let $L^\partial$ be the order dual of $L$, and $X^\partial$ the Priestley space that is order dual to $X$ and is equipped with the same topology. It is straightforward to check that the Priestley dual of $L^\partial$ is $X^\partial$, and that $X$ is co-Esakia iff $X^\partial$ is Esakia. By Esakia duality, $L^\partial$ is a Heyting algebra iff $X^\partial$ is an Esakia space, so $L$ is a co-Heyting algebra iff $X$ is a co-Esakia space. It also follows that $L$ is a bi-Heyting algebra iff $X$ is a bi-Esakia space.
\end{proof}

Our next goal is to show that $X$ is a co-Esakia space iff $\CC(X)$ is a bi-Esakia space. This will provide a necessary and sufficient condition for the G\"odel algebra free over a distributive lattice to be a bi-Heyting algebra. In order to accomplish this, we need to introduce an operation that will help us in computing the upsets of clopen subsets of $\CC(X)$. 

\begin{definition}
Let $X$ be a poset and $\{A_1, \dots, A_n\}$ a finite collection of subsets of $X$. We define the subset $\upup\{A_1,\dots,A_n\}$ of $X$ by induction on $n$. We set 
\[
\upup \varnothing \coloneqq X
\]
and for $n \ge 1$, 
\[
\upup\{A_1,\dots,A_n\} \coloneqq \bigcup_{i=1}^n\up(\upup\{A_1, \dots, \widehat{A_i}, \dots, A_n\} \cap A_i),
\]
where $\{A_1, \dots, \widehat{A_i}, \dots, A_n\}$ denotes the set obtained from $\{A_1, \dots, A_n\}$ by removing $A_i$.
\end{definition}

Since $\upup \{A\} = \up A$, the operation $\upup$ can be thought of as an extension of the operation $\up$ of taking the upset of a subset of $X$. We will use $\upup$ to compute upsets of clopens in $\CC(X)$. But first, we need to prove some properties of this operation.

\begin{lemma}\plabel{lem:upup properties}
Let $X$ be a poset, $x \in X$, and $A_1, \dots, A_n \subseteq X$.
\begin{enumerate}
\item\label[lem:upup properties]{lem:upup properties:item1} $x \in \upup\{A_1,\dots,A_n\}$ iff there is a chain $C$ in $X$ such that $C \subseteq \down x$ and $C \cap A_i \neq \varnothing$ for every $i = 1, \dots, n$.
\item\label[lem:upup properties]{lem:upup properties:item2} If $X$ is a co-Esakia space and $A_1, \dots, A_n$ are clopen in $X$, then $\upup\{A_1,\dots,A_n\}$ is clopen.
\end{enumerate}
\end{lemma}

\begin{proof}
\eqref{lem:upup properties:item1}. 
We prove the claim by induction on $n$. The case $n=0$ is clear. We assume that the claim is true for $n-1$ with $n \ge 1$ and show that it is true for $n$. We first prove the left-to-right implication. Let $x \in \upup\{A_1,\dots,A_n\}$. The definition of $\upup\{A_1,\dots,A_n\}$ implies that there is $i \le n$ such that $x \in \up(\upup\{A_1, \dots, \widehat{A_i}, \dots, A_n\} \cap A_i)$. Then there is $y \le x$ with $y \in \upup\{A_1, \dots, \widehat{A_i}, \dots, A_n\} \cap A_i$. By the induction hypothesis, there exists a chain $K$ in $X$ such that $K \subseteq \down y$ and $K \cap A_j \neq \varnothing$ for every $j \neq i$. Then $C=K \cup \{y\}$ is a chain such that $C \subseteq \down x$ and $C \cap A_i \neq \varnothing$ for every $i$. To prove the other implication, suppose there is a chain $C$ in $X$ such that $C \subseteq \down x$ and $C \cap A_i \neq \varnothing$ for every~$i$. Choose an element $a_i \in C \cap A_i$ for every $i$. Since $C$ is a chain, $K\coloneqq \{ a_1, \dots, a_n\}$ is a chain such that $K \cap A_i \neq \varnothing$ for every $i$. Let $a_j$ be the greatest element of $K$. 
Since $K \subseteq \down a_j$, the induction hypothesis yields that $a_j \in \upup\{A_1, \dots, \widehat{A_j}, \dots, A_n\}$. Thus, $a_j \in \upup\{A_1, \dots, \widehat{A_j}, \dots, A_n\} \cap A_j$. It follows that 
\[
x \in \up (\upup\{A_1, \dots, \widehat{A_j}, \dots, A_n\} \cap A_j) \subseteq \upup\{A_1,\dots,A_n\}.
\]

\eqref{lem:upup properties:item2}. We again argue by induction on $n$. When $n=0$, the set $\upup \varnothing = X$ is clearly clopen. We assume that the claim is true for $n-1$ with $n \ge 1$ and show that it is true for $n$. By the induction hypothesis, $\upup\{A_1, \dots, \widehat{A_i}, \dots, A_n\}$ is clopen for every $i$. Since each $A_i$ is clopen, $\upup\{A_1, \dots, \widehat{A_i}, \dots, A_n\} \cap A_i$ is clopen for every $i$. Thus, $ \up(\upup\{A_1, \dots, \widehat{A_i}, \dots, A_n\} \cap A_i)$ is clopen because it is the upset of a clopen subset and $X$ is a co-Esakia space. So,
\[
\upup\{A_1,\dots,A_n\} = \bigcup_{i=1}^n\up(\upup\{A_1, \dots, \widehat{A_i}, \dots, A_n\} \cap A_i)
\] 
is clopen since it is a finite union of clopens.
\end{proof}

\begin{proposition}\label{prop:upset of clopen}
Let $V, W_1, \dots, W_n$ be subsets of a Priestley space $X$ such that $W_1, \dots, W_n \subseteq V$. Then  
\[
\upc (\Box V \cap \Diamond W_1 \cap \dots \cap \Diamond W_n)
=
\bigcup_{I \subseteq \{1, \dots, n \}} \Big[ \Box (V \cap \upup\{ W_i \mid i \in I\}) \cap \bigcap\{ \Diamond W_j \mid j \in I^c\} \Big],
\]
where in the right-hand side $I$ ranges over all subsets of $\{1, \dots, n\}$ and $I^c \coloneqq \{1, \dots, n\} \setminus I$.
\end{proposition}

\begin{proof}
To show the left-to-right inclusion, assume that $C \in \upc (\Box V \cap \Diamond W_1 \cap \dots \cap \Diamond W_n)$. Then there is $K \in \CC(X)$ such that $K \lec C$, $K \in \Box V$, and $K \in \Diamond W_i$ for every $i=1, \dots, n$. Let $I \subseteq \{1, \dots, n\}$ be such that $i \in I$ iff $C \notin \Diamond W_i$. Then $C \in \Diamond W_j$ for every $j \in I^c$. Thus, $C \in \bigcap\{ \Diamond W_j \mid j \in I^c\}$. It remains to show that $C \in \Box (V \cap \upup\{ W_i \mid i \in I\})$.
By \cref{lem:Box and Diamond properties:item1}, $\Box V$ is an upset, and hence $C \in \Box V$ because $K \lec C$ and $K \in \Box V$. Thus, $C \subseteq V$.
Since $K \in \Diamond W_i$ for each $i= 1, \dots, n$ and $C \notin \Diamond W_i$ for every $i \in I$, it follows that $(K \setminus C) \cap W_i \ne \varnothing$ for every $i \in I$. We have that $K \lec C$ implies $K \setminus C \subseteq \down x$ for every $x \in C$. So, for every $x \in C$ the set $K \setminus C$ is a chain in $X$ such that $K \setminus C \subseteq \down x$ and $(K \setminus C) \cap W_i \ne \varnothing$ for every $i \in I$. Then \cref{lem:upup properties:item1} implies that $C \subseteq \upup\{ W_i \mid i \in I\}$. Therefore, $C \subseteq V \cap \upup\{ W_i \mid i \in I\}$, and hence $C \in \Box (V \cap \upup\{ W_i \mid i \in I\})$.
We have then found $I \subseteq \{1, \dots, n\}$ such that $C \in \Box (V \cap \upup\{ W_i \mid i \in I\}) \cap \bigcap\{ \Diamond W_j \mid j \in I^c\}$. It follows that $C$ belongs to the right-hand side of the claimed equality.

To show the right-to-left inclusion, let $C \in \CC(X)$ be such that 
\[
C \in \Box (V \cap \upup\{ W_i \mid i \in I\}) \cap \bigcap\{ \Diamond W_j \mid j \in I^c\}
\]
for some $I \subseteq \{1, \dots, n\}$. Then $C \in \Box (V \cap \upup\{ W_i \mid i \in I\})$ and $C \in \Diamond W_j$ for every $j \in I^c$.
So, $C \subseteq V \cap \upup\{ W_i \mid i \in I\}$ and $C \cap W_j \ne \varnothing$ for every $j \in I^c$.
Since $C \subseteq \upup\{ W_i \mid i \in I\}$, we have in particular that the least element $m(C)$ of $C$ is in $\upup\{ W_i \mid i \in I\}$. Thus, \cref{lem:upup properties:item1} yields a chain $K$ of $X$ such that $K \subseteq \down m(C)$ and $K \cap W_i \ne \varnothing$ for every $i \in I$. 
By selecting one element from $K \cap W_i$ for each $i$, we can assume that $K$ is finite, $K \subseteq \bigcup \{ W_i \mid i \in I\}$, and $K \cap W_i \ne \varnothing$ for every $i \in I$. Then $K$ is a closed chain because it is a finite chain in $X$. Since $K \subseteq \down m(C)$, we obtain that $K \cup C \in \CC(X)$ and $K \cup C \lec C$. By hypothesis, $\bigcup \{ W_i \mid i \in I\} \subseteq V$. Thus, $K \cup C \subseteq V$, and hence $K \cup C \in \Box V$. Since $K \cap W_i \ne \varnothing$ for every $i \in I$ and $C \cap W_j \ne \varnothing$ for every $j \in I^c$, we have that $K \cup C \in \Diamond W_i$ for every $i=1, \dots, n$. Consequently, $K \cup C \in \Box V \cap \Diamond W_1 \cap \dots \cap \Diamond W_n$, and so $C \in \upc (\Box V \cap \Diamond W_1 \cap \dots \cap \Diamond W_n)$.
\end{proof}

Before we can prove the main results of this section, we need to draw a connection between the topology of a Priestley space $X$ and the topology of a subspace of $\CC(X)$. Let $\max \CC(X)$ denote the subset of $\CC(X)$ consisting of the elements of $\CC(X)$ that are maximal with respect to the order $\lec$. It follows from the definition of $\lec$ that such elements are exactly the chains consisting of a single element of $X$. It is then clear that the map $\varphi \colon X \to \max \CC(X)$ sending $x$ to $\{x\}$ is a bijection. We consider $\max \CC(X)$ equipped with the subspace topology induced by the topology on $\CC(X)$.

\begin{lemma}\label{lem:max CC(X)}
Let $X$ be a Priestley space. Then $\varphi \colon X \to \max \CC(X)$ is a homeomorphism.
\end{lemma}

\begin{proof}
Since we know that $\varphi$ is a bijection, it is sufficient to show that $\varphi$ is open and continuous. Observe that if $V$ is clopen in $X$, then $\varphi[V]=\{ \{x\} \in \max \CC(X) \mid x \in V\}$. Thus, $\varphi[V] = \max \CC(X) \cap \Box V = \max \CC(X) \cap \Diamond V$. Since $X$ is a Stone space, it follows that $\varphi$ is an open map. 
The definition of the topology on $\CC(X)$ implies that the subsets of the form $\max \CC(X) \cap \Box V$ and $\max \CC(X) \cap \Diamond V$, with $V$ clopen in $X$, form a subbasis for the topology on $\max \CC(X)$. Then $\varphi$ is a continuous map because $\varphi^{-1}[\max \CC(X) \cap \Box V] =\varphi^{-1}[\max \CC(X) \cap \Diamond V]=V$ for every clopen subset $V$ of $X$.
\end{proof}

We are now ready to obtain the necessary and sufficient condition for $\CC(X)$ to be a bi-Esakia space. 

\begin{theorem}\label{thm:CCX bi-Esakia}
Let $X$ be a Priestley space. Then $\CC(X)$ is a bi-Esakia space iff $X$ is a co-Esakia space.
\end{theorem}

\begin{proof}
We first show the left-to-right implication. Assume that $\CC(X)$ is a bi-Esakia space and let $V$ be a clopen subset of $X$. We prove that
\begin{equation}\label{eq:max cap diamond V}
\max \CC(X) \cap \upc \Diamond V = \{\{x\} \in \CC(X) \mid x \in \up V\}. 
\end{equation}
Recall that the elements of $\max \CC(X)$ are exactly the chains consisting of a single element of $X$. 
If $x \in X$, then $\{x\} \in \upc \Diamond V$ iff there is $C \in \CC(X)$ such that $C \cap V \neq \varnothing$ and $C \lec \{x\}$. Since $C \lec \{x\}$ iff $x$ is the greatest element of $C$, we have that $\{x\} \in \upc \Diamond V$ implies that $x \in \up V$. Conversely, suppose $x \in \up V$. Then there is $y \in V$ such that $y \le x$, and hence $\{y,x\} \in \Diamond V$ and $\{y,x\} \lec \{x\}$. Thus, $\{x\} \in \upc \Diamond V$. This establishes \cref{eq:max cap diamond V}. Since $\CC(X)$ is a bi-Esakia space and $V$ is clopen, $\upc \Diamond V$ is clopen. It follows from \cref{eq:max cap diamond V} that $\{\{x\} \in \CC(X) \mid x \in \up V\}$ is clopen in $\max \CC(X)$. Then \cref{lem:max CC(X)} yields that $\up V = \varphi^{-1}[\{\{x\} \in \CC(X) \mid x \in \up V\}]$ is clopen. Therefore, $\up V$ is clopen in $X$ for every clopen subset $V$ of $X$, and so $X$ is a co-Esakia space.

To prove the converse implication, assume that $X$ is a co-Esakia space. Since $\CC(X)$ is an Esakia space by \cref{thm:CC in ERS}, it remains to show that the upset of every clopen subset of $\CC(X)$ is clopen. By \cref{lem:facts CC(X):item5}, every clopen of $\CC(X)$ is a finite union of clopens of the form $\Box V \cap \Diamond W_1 \cap \dots \cap \Diamond W_n$ for some clopen subsets $V, W_1, \dots, W_n$ of $X$ such that $W_1, \dots, W_n \subseteq V$. It is then sufficient to show that if $V, W_1, \dots, W_n$ are clopens in $X$ with $W_1, \dots, W_n \subseteq V$, then $\upc(\Box V \cap \Diamond W_1 \cap \dots \cap \Diamond W_n)$ is clopen. By \cref{prop:upset of clopen}, we have to show that 
\[
\bigcup_{I \subseteq \{1, \dots, n \}} \Big[ \Box (V \cap \upup\{ W_i \mid i \in I\}) \cap \bigcap\{ \Diamond W_j \mid j \in I^c\} \Big]
\] 
is clopen in $\CC(X)$. Since $W_1,\dots,W_n$ are clopens, \cref{lem:upup properties:item2} implies that $\upup\{ W_i \mid i \in I\}$ is clopen for every $I \subseteq \{1, \dots, n \}$. So, $\Box (V \cap \upup\{ W_i \mid i \in I\})$ is clopen in $\CC(X)$ because $V$ is clopen in $X$. Moreover, $\Diamond W_i$ is clopen in $\CC(X)$ for every $i =1, \dots, n$ because each $W_i$ is clopen in $X$. This concludes the proof because finite unions and intersections of clopens are clopen.
\end{proof}

As an immediate consequence of \cref{thm:free Godel over L,thm:CCX bi-Esakia,prop:Esakia duality for cobiHeyting}, we obtain a necessary and sufficient condition for a G\"odel algebra free over a distributive lattice to be a bi-Heyting algebra.

\begin{theorem}\label{thm:free over L biHey iff L coHey}
Let $L$ be a distributive lattice and $G$ the G\"odel algebra free over $L$. Then $G$ is a bi-Heyting algebra iff $L$ is a co-Heyting algebra.
\end{theorem}

\begin{remark}\plabel{remark:CC(X) co and bi}
\hfill\begin{enumerate}
\item\label[remark:CC(X) co and bi]{remark:CC(X) co and bi:item1}
Co-Heyting algebras can be equivalently defined as distributive lattices equipped with a binary operation of co-implication $\leftarrow$ satisfying the property
\[
a \leftarrow b \le c \iff a \le b \vee c.
\]
Co-Heyting algebras form a variety in the signature $(\wedge, \vee, \leftarrow, 0, 1)$ and a lattice homomorphism between co-Heyting algebras is called a co-Heyting algebra homomorphism if it also preserves the co-implication. Co-Heyting algebra homomorphisms correspond to continuous map between co-Esakia spaces satisfying $f[\down x] = \down f(x)$ for every $x$ in the domain.

Let $X$ be a Priestley space. It is straightforward to verify that the continuous map $m \colon \CC(X) \to X$, that sends each $C \in \CC(X)$ its least element $m(C)$, satisfies $m[\downc C] = \down m(C)$. If $X$ is a co-Esakia space, then \cref{thm:CCX bi-Esakia} implies that $\CC(X)$ is a bi-Esakia space, and hence a co-Esakia space.  Therefore, if $L$ is a co-Heyting algebra and $G$ the G\"odel algebra free over the distributive lattice $L$ via 
$e \colon L \to G$, then $e$ is a co-Heyting algebra homomorphism.

\item \label[remark:CC(X) co and bi]{remark:CC(X) co and bi:item2}
Let $X$ be a Priestley space and $\lec'$ the order on $\CC(X)$ given by $C_1 \lec' C_2$ iff $C_1 \subseteq C_2$ and $C_1$ is a downset in $C_2$.  Then $(\CC(X), \lec')$ is the order dual of $(\CC(X^\partial), \lec)$, where $X^\partial$ is the order dual of $X$. Thus, \cref{thm:CCX bi-Esakia} implies that $(\CC(X), \lec')$ is a bi-Esakia space iff $X$ is an Esakia space. Moreover, what we observed in \eqref{remark:CC(X) co and bi:item1} yields that the map $M \colon (\CC(X), \lec') \to (X, \le)$ sending each $C \in \CC(X)$ to its greatest element $M(C) \in X$ is a continuous p-morphism. It follows that every Esakia space $X$ is the image under a continuous p-morphism of the bi-Esakia space $(\CC(X), \lec')$, which is a forest; i.e., a disjoint union of trees.
This is a natural generalization of the standard unraveling construction that ``unfolds'' a rooted poset into a tree (see, e.g., \cite[Thm.~2.19]{CZ97}). 
\end{enumerate}
\end{remark}

Recall that if $S$ is a set, then $\two^S$ denotes the Priestley space equipped with the product topology and the pointwise order induced by the $2$-element chain $\two$ with the discrete topology. Thus, the topology on $\two^S$ is generated by the subbasis $\{ U_s, D_s \mid s \in S\}$, where
\[
U_s \coloneqq \{ (a_i)_{i \in S} \mid a_s=1\} \qquad \text{and} \qquad D_s \coloneqq \{ (a_i)_{i \in S} \mid a_s=0\}.
\]
The subsets $U_s$ and $D_s$ are clopen upsets and clopen downsets of $\two^S$ for every $s \in S$, respectively. The following fact is well known. Due to the lack of a reference, we provide its proof.

\begin{lemma}\label{lem:free dist bi-Heyting}
Free distributive lattices are bi-Heyting algebras.
\end{lemma}

\begin{proof}
Let $S$ be a set. By \cref{prop:free dist latt}, the free distributive lattice over $S$ is dual to the Priestley space $\two^S$. By duality, it is then sufficient to show that $\two^S$ is a bi-Esakia space. Since $\{ U_s, D_s \mid s \in S\}$ is a subbasis consisting of clopen subsets, to prove that $\two^S$ is a bi-Esakia space, it is sufficient to show that $\down (U_{s_1} \cap \dots \cap U_{s_n} \cap D_{t_1} \cap \dots \cap D_{t_m})$ and $\up (U_{s_1} \cap \dots \cap U_{s_n}  \cap D_{t_1} \cap \dots \cap D_{t_m})$ are clopen for every $s_1, \dots, s_n,t_1,\dots,t_m \in S$. Since $U_s \cap D_s= \varnothing$ for every $s \in S$, we can assume that $s_i \ne t_j$ for every $i \le n$ and $j \le m$.
It is straightforward to check that
\begin{align*}
\down (U_{s_1} \cap \dots \cap U_{s_n} \cap D_{t_1} \cap \dots \cap D_{t_m}) & = D_{t_1} \cap \dots \cap D_{t_m}, \text{ and}\\
\up (U_{s_1} \cap \dots \cap U_{s_n} \cap D_{t_1} \cap \dots \cap D_{t_m}) & = U_{s_1} \cap \dots \cap U_{s_n} .
\end{align*}
Thus, $\up V$ and $\down V$ are clopen for every clopen subset $V$ of $\two^S$. Therefore, $\two^S$ is a bi-Esakia space.
\end{proof}

\begin{theorem}
Free G\"odel algebras are bi-Heyting algebras.
\end{theorem}

\begin{proof}
Let $G$ be the G\"odel algebra free over a set $S$. The $G$ is also free over the distributive lattice $L$ that is free over $S$. By \cref{lem:free dist bi-Heyting}, $L$ a bi-Heyting algebra, and hence a co-Heyting algebra. Therefore, \cref{thm:free over L biHey iff L coHey} yields that $G$ is a bi-Heyting algebra.
\end{proof}

We end the section by showing that an analogue of the previous theorem does not hold for $\GA_n$-algebras. We first need to prove the following technical fact.

\begin{lemma}\label{lem:clopen CCn contain chains size n}
If $S$ is an infinite set, then every nonempty clopen subset of $\CC_n(\two^S)$ contains a chain of $\two^S$ of size $n$.
\end{lemma}

\begin{proof}
Since $\{ U_s, D_s \mid s \in S\}$ is a subbasis for $\two^S$, every clopen of $\two^S$ is a finite union of finite intersections of subsets of the form $U_s$ and $D_s$. Thus, for each clopen $V$ of $\two^S$ there is a finite subset $S_V \subseteq S$ such that $V=\pi_{S_V}^{-1}[V']$ for some $V' \subseteq \two^{S_V}$, where $\pi_{S_V} \colon \two^S \to \two^{S_V}$ is the projection mapping each $(a_i)_{i \in S}$ to its subsequence $(a_i)_{i \in S_V}$. It follows that if $\overline{a}=(a_i)_{i \in S}$ and $\overline{b}=(b_i)_{i \in S}$ are elements of $\two^S$ such that $\pi_{S_V}(\overline{a})= \pi_{S_V}(\overline{b})$, then $\overline{a} \in V$ iff $\overline{b} \in V$.

Let $\mathcal{V}$ be a nonempty clopen of $\CC_n(\two^S)$. We show that $\mathcal{V}$ contains a chain of size $n$.  By \cref{lem:facts CC(X):item5}, we can assume that $\mathcal{V}=(\Box V \cap \Diamond W_1 \cap \dots \cap \Diamond W_m) \cap \CC_n(\two^S)$ for some $V,W_1,\dots,W_m$ clopens of $\two^S$. Let $S_\mathcal{V}\coloneqq S_V \cup S_{W_1} \cup \dots \cup S_{W_m} \subseteq S$. Thus, if $C_1,C_2 \in \CC_n(\two^S)$ are such that $\pi_{S_\mathcal{V}}[C_1]=\pi_{S_\mathcal{V}}[C_2]$, then what we observed above yields that $C_1 \in \mathcal{V}$ iff $C_2 \in \mathcal{V}$. 
Since $\mathcal{V}$ is nonempty, there is $C \in \mathcal{V}$. Let $\overline{a}^1, \dots, \overline{a}^k$ be the elements of $C$ with 
$\overline{a}^1 < \dots < \overline{a}^k$, and let
$\overline{a}^j=(a_i^j)_{i \in S}$ for every $j=1, \dots, k$. If $n=k$, then $C$ is a chain of size $n$ in $\mathcal{V}$ and we are done. Assume $n \ne k$. Then $k < n$. We define $\overline{b}^1, \dots, \overline{b}^k \in \two^S$ with $\overline{b}^j=(b_i^j)_{i \in S}$ for every $j$ by setting $b_i^j=a_i^j$ if $i \in S_{\mathcal{V}}$ and $b_i^j=0$ otherwise. 
Then $\overline{b}^1 \leq \dots \leq \overline{b}^k$ and 
$C'=\{\overline{b}^1 , \dots, \overline{b}^k\}$ is a chain of size smaller or equal to $k$. By the definition of the $\overline{b}^j$'s, we obtain that $\pi_{S_\mathcal{V}}[C]=\pi_{S_\mathcal{V}}[C']$. Let $h$ be the size of $C'$ and pick $s_1, \dots, s_{n-h}$ distinct elements of $S \setminus S_{\mathcal{V}}$, which exist because $S$ is infinite and $S_{\mathcal{V}}$ is finite.  Define $\overline{c}^1, \dots, \overline{c}^{n-h} \in \two^S$ with $\overline{c}^j=(c_i^j)_{i \in S}$ by setting
\[
c_i^j=
\begin{cases}
1 &\text{if } i \in \{s_1, \dots, s_j\},\\
0 &\text{if } i \in \{s_{j+1}, \dots, s_{n-h}\},\\
b_i^k &\text{otherwise.}
\end{cases}
\]
Then $\overline{b}^k < \overline{c}^1 < \dots < \overline{c}^{n-h}$ and $\pi_{S_\mathcal{V}}(\overline{c}^j)= \pi_{S_\mathcal{V}}(\overline{b}^k)$ for every $j$. Therefore, $C''\coloneqq C' \cup \{ \overline{c}^1, \dots, \overline{c}^{n-h}\}$
is a chain of $\two^S$ of size $n$ that belongs to $\mathcal{V}$ because $\pi_{S_\mathcal{V}}[C'']=\pi_{S_\mathcal{V}}[C']=\pi_{S_\mathcal{V}}[C]$ and $C \in \mathcal{V}$.
\end{proof}

\begin{theorem}
If $n \ge 2$, then every $\GA_n$-algebra free over an infinite set is not a bi-Heyting algebra. 
\end{theorem}

\begin{proof}
By \cref{thm:free GAn over set,prop:Esakia duality for cobiHeyting}, it is enough to show that $\CC_n(\two^S)$ is not a bi-Esakia space whenever $S$ is an infinite set. So, we need to exhibit a clopen subset $\mathcal{V}$ of $\CC_n(\two^S)$ such that $\upc \mathcal{V}$ is not clopen in $\CC_n(\two^S)$. Note that  $\upc \mathcal{V} \subseteq \CC_n(\two^S)$ for every $\mathcal{V} \subseteq \CC_n(\two^S)$ because $\CC_n(\two^S)$ is an upset of $\CC(\two^S)$.

Fix $s \in S$ and let $\mathcal{V} = (\Diamond D_s) \cap \CC_n(\two^S)$. Assume, with a view to contradiction, that $\upc \mathcal{V}$ is clopen in $\CC_n(\two^S)$. Define two elements $\overline{a} = (a_i)_{i \in S}$ and $\overline{b} = (b_i)_{i \in S}$ of $\two^S$ by setting $a_i=0$ for every $i \in S$ and $b_i=1$ for $i = s$ and $b_i=0$, otherwise. Then $\overline{a} \in D_s$, $\overline{b} \notin D_s$, and $\overline{a} < \overline{b}$. Thus, $\{ \overline{a}, \overline{b}\} \in \mathcal{V}$, and hence $\{\overline{b}\}$ is a closed chain in $\two^S$ that belongs to $\upc \mathcal{V} \setminus \mathcal{V}$. It follows that the subset $\upc \mathcal{V} \setminus \mathcal{V}$ of $\CC_n(\two^S)$ is nonempty and it is clopen by our assumption. By \cref{lem:clopen CCn contain chains size n}, $\upc \mathcal{V} \setminus \mathcal{V}$ must contain a chain $C$ of size $n$. Because $C \in \upc \mathcal{V}$, there is $K \in \mathcal{V}$ such that $K \lec C$. Since $C$ has size $n$ and $K \in \CC_n(\two^S)$, we obtain that $C=K$ because $K \lec C$. Then $C=K \in \mathcal{V}$, but this contradicts that $C \in \upc \mathcal{V} \setminus \mathcal{V}$. Therefore, $\upc \mathcal{V}$ is not clopen. This shows that $\CC_n(\two^S)$ is not a bi-Esakia space.
\end{proof}

\section{Comparison with the step-by-step method}\label{sec:comparison stepbystep}

In this section we compare the dual description of free G\"odel algebras obtained in \cref{sec:free Godel} with the one resulting from the step-by-step method. The step-by-step method was introduced in \cite{Ghi92} to study Heyting algebras free over finite distributive lattices and has been extended in \cite{Alm24} to Heyting algebras free over distributive lattices of any cardinality.
We briefly recall the description of free G\"odel algebras obtained in \cite[Sec.~6.3]{Alm24} utilizing the step-by-step approach.

Let $X$ be a Priestley space. Topologize the set of nonempty closed chains $\CC(X)$ of $X$ with the Vietoris topology as we did in \cref{sec:free Godel}. However, instead of equipping $\CC(X)$ with the partial order $\lec$, equip it with the reverse inclusion order $\supseteq$. It follows from \cite[Sec.~6.3]{Alm24} that $(\CC(X), \supseteq)$ is a Priestley space. Observe that $(\CC(X), \supseteq)$ is not a root system in general. To simplify notation, in what follows we denote the Priestley space $(\CC(X), \supseteq)$ by $Y$.
To describe how to obtain an Esakia root system from $Y$, we need to introduce the notion of $m$-open element of $\CC(Y)$. Note that we are now considering closed chains in $Y$, whose elements are themselves closed chains in $X$.

\begin{definition}
Let $m \colon Y \to X$ be the map that sends a nonempty closed chain of $X$ to its least element. We say that $\mathcal{C} \in \CC(Y)$ is \emph{$m$-open} provided that for every $C_1 \in \mathcal{C}$ and $C_2 \in Y$ with $C_1 \supseteq C_2$, there is $C_3 \in \mathcal{C}$ such that $C_1 \supseteq C_3$ and  $m(C_2)=m(C_3)$. 
\end{definition}

Let $Z = \{ \mathcal{C} \in \CC(Y) \mid \mathcal{C} \text{ is $m$-open}\}$ and equip $Z$ with the subspace topology induced by the Vietoris topology on $\CC(Y)$.  The following theorem, which is a consequence of  \cite[Thms.~6.11 and~6.15]{Alm24}, provides an alternative dual description of the free G\"odel algebra over a distributive lattice.

\begin{theorem}
The ordered space $(Z,\supseteq)$ is an Esakia root system and $(Z,\supseteq)^*$ is a G\"odel algebra free over the distributive lattice $X^*$.
\end{theorem}

Since both $(\CC(X), \lec)$ and $(Z, \supseteq)$ are dual to G\"odel algebras free over $X^*$ and free algebras are unique up to isomorphism, it follows that the two Esakia root systems must be isomorphic. We sketch a direct proof of the existence of this isomorphism.

\begin{theorem}\label{thm:CCX iso to Z}
The Esakia root systems $(\CC(X), \lec)$ and $(Z, \supseteq)$ are isomorphic.
\end{theorem}

\begin{proof}[Sketch of the proof.]
Recall that the elements of $Z$ are the $m$-open nonempty closed chains of $Y=(\CC(X), \supseteq)$. It can be shown that if $C \in \CC(X)$, then $\upc C$ is an element of $Z$, and that any element of $Z$ is $\upc C$ for some $C \in \CC(X)$. 
Sending each $C \in \CC(X)$ to $\upc C \in Z$ defines a bijection between $\CC(X)$ and $Z$. It turns out that this map is an isomorphism of Esakia root systems. 
\end{proof}

Proving the missing steps in the sketch of the proof of \cref{thm:CCX iso to Z} requires a nontrivial effort since the definition of $Z$ is quite involved: $Z$ is equipped with the Vietoris topology induced by the topology on $\CC(X)$, which in turn is the Vietoris topology induced by $X$.
It is for this reason that, instead of deriving \cref{thm:free Godel over L} from the results of \cite{Alm24} and \cref{thm:CCX iso to Z}, we opted to provide a more direct and independent proof in \cref{sec:free Godel}.

We end this final section by turning our attention to the Priestley space $(\CC(X), \supseteq)$ that played a fundamental role in the step-by-step approach. When $X$ is a finite poset, the order dual of $(\CC(X), \supseteq)$ is the nerve of $X$, which has applications in polyhedral geometry (see, e.g., \cite[p.~388]{BMMP18} and the references therein). \cref{fig:twotwo and CCtwotwo} depicts the poset $\two \times \two$, where $\two$ is the $2$-element chain, and the two partial orders $\lec$ and $\supseteq$ on $\CC(\two \times \two)$. The solid lines denote the partial order $\lec$ and the dotted lines show the relations that need to be added to $\lec$ to obtain $\supseteq$. Note that \cref{thm:free Godel over S} yields that $(\CC(\two \times \two), \lec)$ is the Esakia root system dual to the G\"odel algebra free over $2$ generators. 

\begin{figure}[!h]
\begin{tikzpicture}[scale=0.8]
	\begin{pgfonlayer}{nodelayer}
		\node [style=black dot] (0) at (-3, 1.75) {};
		\node [style=black dot] (1) at (-3, 0) {};
		\node [style=black dot] (2) at (0, 1.75) {};
		\node [style=black dot] (3) at (1.5, 1.75) {};
		\node [style=black dot] (4) at (3, 1.75) {};
		\node [style=black dot] (5) at (-1.5, 0) {};
		\node [style=black dot] (6) at (0, 0) {};
		\node [style=black dot] (7) at (1.5, 0) {};
		\node [style=black dot] (8) at (3, 0) {};
		\node [style=black dot] (9) at (-1.5, -1.75) {};
		\node [style=black dot] (10) at (1.5, -1.75) {};
		\node [style=black dot] (11) at (-8.5, 1.75) {};
		\node [style=black dot] (12) at (-6.75, 0) {};
		\node [style=black dot] (13) at (-10.25, 0) {};
		\node [style=black dot] (14) at (-8.5, -1.75) {};
		\node [style=none] (15) at (-8.5, -2.625) {};
		\node [style=none] (16) at (0, -2.625) {};
		\node [style=none] (17) at (-8.5, -2.625) {$\two \times \two$};
		\node [style=none] (18) at (0, -2.625) {$\CC(\two \times \two)$};
	\end{pgfonlayer}
	\begin{pgfonlayer}{edgelayer}
		\draw (2) to (5);
		\draw (2) to (7);
		\draw (2) to (6);
		\draw (5) to (9);
		\draw (7) to (10);
		\draw (4) to (8);
		\draw (0) to (1);
		\draw [style=dotted] (10) to (8);
		\draw [style=dotted] (7) to (4);
		\draw [style=dotted] (5) to (0);
		\draw [style=dotted] (9) to (1);
		\draw [style=dotted] (6) to (3);
		\draw [style=dotted] (3) to (8);
		\draw [style=dotted] (10) to (6);
		\draw [style=dotted] (6) to (9);
		\draw [style=dotted] (1) to (3);
		\draw (11) to (12);
		\draw (12) to (14);
		\draw (14) to (13);
		\draw (13) to (11);
	\end{pgfonlayer}
\end{tikzpicture}
\caption{The poset $\two \times \two$ and the set $\CC(\two \times \two)$ with the partial orders $\lec$ and $\supseteq$.}\label{fig:twotwo and CCtwotwo}
\end{figure}

Let $G$ be the G\"odel algebra free over $L=X^*$ via $e \colon L \to G$. Then $G \cong (\CC(X), \lec)^*$ by \cref{thm:free Godel over L}.
Since $\supseteq$ extends $\lec$, the identity map $\id_{\CC(X)} \colon (\CC(X), \lec) \to (\CC(X), \supseteq)$ is a continuous order-preserving map between Priestley spaces. Then $(\CC(X), \supseteq)^*$ embeds into $G$ because onto $\Pries$-morphisms correspond to embeddings in $\DL$ (see, e.g., \cite[Thm.~11.31]{DP02}). 
Let $L' \subseteq G$ be the subalgebra of $G$ that is the image of the embedding of $(\CC(X), \supseteq)^*$ into $G$.
It is a consequence of \cite[Sec.~6.3]{Alm24} that $L'$ is the bounded sublattice of $G$ generated by the subset $\{e(a) \to e(b) \mid a,b \in L\}$.
Intuitively, $L'$ is the result of the first step in the step-by-step construction of $G$ and is obtained by ``freely adding'' to $L$ implications between its elements.\footnote{See \cite[Sec.~3.1]{Alm24} for an intuitive explanation of the ideas behind the step-by-step construction.} 
Since $\id_{\CC(X)} \colon (\CC(X), \lec) \to (\CC(X), \supseteq)$ is a bijection, it follows from Priestley duality that there is a bijection between the sets of prime filters of $G$ and $L'$ that sends a prime filter $P$ of $G$ to the prime filter $P \cap L'$ of $L'$. Note that this correspondence preserves inclusions but does not necessarily reflects them.

We end the section by characterizing the clopen upsets of $(\CC(X), \supseteq)$, which correspond to the elements of the sublattice $L'$ of $G \cong (\CC(X), \lec)^*$, as we observed in the previous paragraph. By \cref{thm:free Godel over L}, the elements of $G$ of the form $e(a)$ with $a \in L$ correspond to the clopen upsets of $(\CC(X), \lec)$ of the form $\Box U$ with $U$ a clopen upset of $X$. We first describe the implications between such elements in $(\CC(X), \lec)^*$.

\begin{proposition}\label{prop:implication box clopup}
If $U_1,U_2$ are clopen upsets of $X$, then $\Box U_1 \to \Box U_2=\Box((X \setminus U_1) \cup U_2)$.
\end{proposition}

\begin{proof}
Since $(\CC(X), \lec)$ is an Esakia space, the implication in $(\CC(X), \lec)^*$ is given by
\[
\Box U_1 \to \Box U_2 = \CC(X) \setminus \downc (\Box U_1 \setminus \Box U_2). 
\]
We first show that $\downc (\Box U_1 \setminus \Box U_2) = \Diamond (U_1 \setminus U_2)$. Let $C \in \CC(X)$. Then $C \in \downc (\Box U_1 \setminus \Box U_2)$ iff there is $K \in \CC(X)$ such that $C \lec K$, $K \in \Box U_1$, and $K \notin \Box U_2$. The existence of such a $K$ is equivalent to the existence of $x \in C$ such that $C \cap \up x \subseteq U_1$ and $C \cap \up x \nsubseteq U_2$. Since $U_1$ and $U_2$ are upsets, it follows that $C \in \downc (\Box U_1 \setminus \Box U_2)$ iff there is $x \in C$ such that $x \in U_1 \setminus U_2$. Thus, $\downc (\Box U_1 \setminus \Box U_2) = \Diamond (U_1 \setminus U_2)$. Then \cref{lem:facts CC(X):item2} implies that $\Box U_1 \to \Box U_2 = \CC(X) \setminus \Diamond (U_1 \setminus U_2) = \Box ((X\setminus U_1) \cup U_2)$.
\end{proof}

We end this last section of the paper with a theorem characterizing the elements of the distributive lattice $(\CC(X), \supseteq)^*$ isomorphic to $L'$.

\begin{theorem}
The clopen upsets of $(\CC(X), \supseteq)$ are the subsets of $\CC(X)$ of the form $\Box V_1 \cup \dots \cup \Box V_n$ with $V_1, \dots, V_n$ clopen in $X$. 
\end{theorem}

\begin{proof}
By what we observed before \cref{prop:implication box clopup}, $(\CC(X), \supseteq)^*$ is the sublattice of $(\CC(X), \lec)^*$ generated by the elements of the form $\Box U_1 \to \Box U_2$ with $U_1,U_2$ clopen upsets of $X$. Thus, \cref{prop:implication box clopup} implies that the elements of $(\CC(X), \supseteq)^*$ are finite unions of finite intersections of elements of the form $\Box((X \setminus U_1) \cup U_2)$, with $U_1$ and $U_2$ clopen upsets of $X$. 
Since $\Box$ commutes with finite intersections by \cref{lem:facts CC(X):item1}, we obtain that any clopen upset of $(\CC(X), \supseteq)$ is a finite union of subsets of the form $\Box V$ with $V$ clopen in $X$. Conversely, it is straightforward to check that $\Box V$ is a clopen upset of $(\CC(X), \supseteq)$ for every clopen subset $V$ of $X$. Therefore, every finite union of subsets of the form $\Box V$ with $V$ clopen is a clopen upset of $(\CC(X), \supseteq)$.
\end{proof}

\subsection*{Acknowledgements}

I would like to thank Guram Bezhanishvili for pointing out some relevant references. I am also grateful to Vincenzo Marra and Tommaso Moraschini for their comments and insightful conversations.

\end{document}